\documentclass{amsart}
\usepackage{amsthm}
\usepackage{amstext}
\usepackage{amssymb}
\usepackage{amsfonts}

\numberwithin{equation}{section}
\numberwithin{figure}{section}

\newtheorem{theorem}{Theorem}[section]
\newtheorem{lemma}[theorem]{Lemma}
\theoremstyle{definition}
\newtheorem{definition}[theorem]{Definition}

\theoremstyle{remark}
\newtheorem{remark}[theorem]{Remark}

\numberwithin{equation}{section} \theoremstyle{plain}

\newtheorem{corollary}[theorem]{Corollary}

\newtheorem{proposition}[theorem]{Proposition}

\copyrightinfo{2001}{enter name of copyright holder}

\newcommand{\set}[1]{\left\{#1\right\}}

\input{tcilatex}

\begin{document}

\title[Rigidity of warped product Einstein manifolds]{Warped product Einstein metrics over spaces with constant scalar curvature}

\author{Chenxu He}
\address{14 E. Packer Ave\\
Dept. of Math, Lehigh University \\
Christmas-Saucon Hall\\
Bethlehem, PA 18015}
\email{he.chenxu@lehigh.edu}
\urladdr{http://sites.google.com/site/hechenxu/}

\author{Peter Petersen}
\address{520 Portola Plaza\\
Dept of Math UCLA\\
Los Angeles, CA 90095}
\email{petersen@math.ucla.edu}
\urladdr{http://www.math.ucla.edu/\symbol{126}petersen}
\thanks{The second author was supported in part by NSF-DMS grant 1006677}

\author{William Wylie}
\address{209 S 33rd St \\
Dept. of Math, University of Pennsylvania\\
Philadelphia, PA 19104.}
\email{wylie@math.upenn.edu}
\urladdr{http://www.math.upenn.edu/\symbol{126}wylie}
\thanks{The third author was supported in part by NSF-DMS grant 0905527}


\begin{abstract}
In this paper we study warped product Einstein metrics over spaces with constant scalar curvature. We call such a manifold rigid if the universal cover of the base is Einstein or is isometric to a product of Einstein manifolds. When the base is three dimensional and the dimension of the fiber is greater than one we show that the space is always rigid. We also exhibit examples of solvable four dimensional Lie groups that can be used as the base space of non-rigid warped product Einstein metrics showing that the result is not true in dimension greater than three. We also give some further natural curvature conditions that characterize the rigid examples in higher dimensions.
\end{abstract}

\maketitle

\section{Introduction}

A $(\lambda,n+m)$-Einstein manifold $(M^{n},g,w)$ is a complete Riemannian manifold, possibly with boundary, and a smooth function $w$ on $M$ satisfying:
\begin{eqnarray}
\mathrm{Hess}w & = & \frac{w}{m}(\mathrm{Ric}-\lambda g)\notag\\
w & > & 0\mbox{ on }\mathrm{int}(M),\label{eqnlambdan+m}\\
w & = & 0\mbox{ on }\partial M.\notag
\end{eqnarray}
When $m=1$, we make the additional assumption that $\Delta w=-\lambda w$.

$(\lambda,n+m)$-Einstein metrics are also called \emph{$m$-Quasi Einstein} manifolds and the case where $\partial M$ is empty was studied earlier in \cite{CaseShuWei} and \cite{KimKim}. We also
studied this equation in \cite{HPWLcf} and showed that many of the results from these earlier works generalize to the case where the boundary is non-empty.

The $(\lambda,n+m)$-Einstein equation has a natural geometric interpretation. Namely, if $m>1$ is an integer, then $(M^{n},g,w)$ is a $(\lambda,n+m)$-Einstein manifold if and only if there is a smooth $(n+m)$-dimensional warped product Einstein metric (with no boundary) with base space $M$ (see \cite[Proposition 1.1]{HPWLcf}). Thus we can study warped product Einstein metrics by analyzing the $(\lambda,n+m)$-Einstein equation on the lower dimensional base space. While the motivation requires $m$ to be an integer, there is no reason to restrict to this case in any of our results. When $m=1$, solutions are also called the \emph{static metrics} and have been studied thoroughly for their connection to general relativity and the positive mass theorem \cite{Corvino}. The equation is also closely related to considerations in optimal transport \cite{Villani} and comparison geometry\cite{WeiWylie}.

Many Einstein metrics can be constructed as warped products. In fact the first non-trivial example of an Einstein metric, the Schwarzschild metric, is a $4$-dimensional doubly warped product metric on $\mathbb{R}^{2}\times\mathbb{S}^{2}$. In our context it can be viewed in two ways, either as a $(0,2+2)$-Einstein metric on $\mathbb{R}^{2}$, or as a $(0,3+1)$-Einstein metric on $[0,\infty)\times\mathbb{S}^{2}$. Much more recently C. B\"{o}hm constructed interesting warped product Einstein metrics on spheres and product of spheres. These examples give rotationally symmetric $(\lambda,n+m)$-Einstein metrics on hemispheres and spheres respectively when $n=3,4,5,6,7$ (see \cite{Bohm1}). Also see \cite{LvPagePope,Bohm2} for more examples.

The B\"{o}hm metrics are in contrast to the solutions to the analogous \emph{gradient Ricci soliton} equation,
\begin{equation*}
\mathrm{Ric}+\mathrm{Hess}f=\lambda g.
\end{equation*}
There is a well known classification theorem of three dimensional gradient Ricci solitons with $\lambda>0$ following from the works of Ivey\cite{Ivey}, Hamilton\cite{Hamilton}, and Perelman\cite{Perelman}. Thus the B\"{o}hm metrics show that there are more examples of $(\lambda,3+m)$-Einstein metrics than gradient Ricci solitons. It is then natural to ask whether one can classify $(\lambda,3+m)$-Einstein metrics under additional natural curvature assumptions. Before we state our first result in this direction, we require a definition.
\begin{definition}
A $(\lambda,n+m)$-Einstein manifold $(M,g,w)$ is called \emph{rigid} if it is Einstein or its universal cover is a product of Einstein manifolds.
\end{definition}

There is an explicit list of the possible rigid examples, in particular we will see that there are at most two different Einstein constants in the product, one of which is $\lambda$, see Section 2. Our first result is that, in dimension three, constant scalar curvature characterizes rigidity.
\begin{theorem} \label{thmdim3}
A $(\lambda,3+m)$-Einstein manifold with $m>1$ has constant scalar curvature if and only if it is rigid. In particular, if the manifold has no boundary, then it is a quotient of $\mathbb{S}^{3}$, $\mathbb{S}^{2}\times\mathbb{R}$, $\mathbb{R}^{3}$, $\mathbb{H}^{2}\times\mathbb{R}$, or $\mathbb{H}^{3}$ with the standard metrics.
\end{theorem}

\begin{remark}
In \cite{Seshadri} H. Seshadri considered compact $(\lambda,3+1)$-Einstein manifold where the total space admits a non-trivial circle action. Under certain further conditions on the circle action, he showed that the total space is either the standard $\mathbb{S}^{4}$ with the base $3$-disk $\mathbb{D}^{3}$ or the Riemannian product $\mathbb{S}^{2}\times\mathbb{S}^{2}$ with the base $\left[-\frac{\pi}{2},\frac{\pi}{2}\right]\times\mathbb{S}^{2}$.
\end{remark}

Note, on the other hand, that every $(\lambda,n+1)$-Einstein metric has constant scalar curvature. While constant scalar curvature may appear to be a strong assumption, we also exhibit examples showing that this theorem is not true in higher dimensions.

\begin{theorem} \label{thmhomogeneousnonrigid}
For each $m>0,$ there are 4-dimensional solvable Lie groups with left invariant $\left(\lambda,4+m\right)$-Einstein metrics that are not rigid.
\end{theorem}

\begin{remark}
These examples are also in stark contrast to the gradient Ricci soliton case, where all homogeneous gradient Ricci solitons are rigid. As $m$ goes to infinity these examples converge to a Riemannian
product $\mathbb{R}\times H^{3}$ where $H$ is a homogeneous (non-gradient) Ricci soliton, see Theorem \ref{thmnonrigidconvergence}.
\end{remark}

\begin{remark}
We also note that the left invariant metrics we construct on solvable Lie groups are not solvsoliton metrics unless they are isometric to $\mathbb{R}\times H$ or are rigid.
\end{remark}

We are then led to the question of whether we can classify $(\lambda,n+m)$-Einstein metrics with constant scalar curvature in higher dimensions under natural additional assumptions. This problem is also addressed in \cite{CaseShuWei} where it is shown that every compact $(\lambda,n+m)$-Einstein metric with constant scalar curvature which does not have boundary must be a trivial $\lambda$-Einstein metric. They also show that when $\lambda=0$ the metric must be Ricci flat, our next result is the extension of this result to the non-empty boundary case.

\begin{theorem} \label{thmuniquenesslambdazero}
Suppose $(M,g,w)$ is a complete $(0,n+m)$-Einstein manifold with constant scalar curvature, then $(M,g)$ is either Ricci flat without boundary or it is isometric to $(F\times[0,\infty),g_{F}+\mathrm{d}r^{2})$ where $(F,g_{F})$ is Ricci flat and $w=w(r)$ is a linear function.
\end{theorem}

Classification theorems for gradient Ricci solitons with constant scalar curvature were also considered by the last two authors in \cite{PWclassification}. They show that a gradient Ricci soliton is rigid if it has constant scalar curvature and all the sectional curvatures in the direction of the gradient of the potential function are zero. Our next result appears to be a stronger result for $(\lambda,n+m)$-Einstein metrics.

\begin{theorem} \label{thmscaln-1}
Suppose $(M,g,w)$ is a complete $(\lambda,n+m)$-Einstein manifold with constant scalar curvature and $\lambda>0(<0)$ , if $\mathrm{Ric}(\nabla w,\nabla w)\leq0(\geq0)$, then $M$ is isometric to the product $\mathbb{R}\times N$ where $N$ is a $\lambda$-Einstein manifold.
\end{theorem}

\begin{remark}
It turns out that the assumption about Ricci curvature is equivalent to the scalar curvature being bounded between $n\lambda$ and $(n-1)\lambda$. Thus this theorem can also be viewed as a generalization of Proposition 3.6 in \cite{CaseShuWei} which states that for $\lambda<0$ if the scalar curvature is constant then it is bounded below by $n\lambda$, and is equal to $n\lambda$ if and only if the metric is $\lambda$-Einstein.
\end{remark}

This theorem shows that only a limited number of rigid $(\lambda,n+m)$-Einstein metrics have vanishing sectional curvatures in the ``radial" direction. This is in contrast with the gradient soliton case where all rigid examples have vanishing radial curvatures. Our next theorem is a characterization of all rigid $(\lambda,n+m)$-Einstein metrics in terms of scalar curvature and the sectional curvatures in the radial direction.

\begin{theorem} \label{thmradialQ}
Suppose $(M,g,w)$ is a complete simply connected non-trivial $(\lambda,n+m)$-Einstein manifold with constant scalar curvature and $\lambda\neq0$. If the Riemann curvature tensor does not grow exponentially, and the radial sectional curvature satisfies
\begin{equation}
R(X,\nabla w,\nabla w,X)=\frac{|\nabla
w|^{2}}{m}\left(\lambda|X|^{2}-\mathrm{Ric}(X,X)\right),\label{eqnradialseccurvature}
\end{equation}
then the manifold is rigid.
\end{theorem}

\begin{remark}
We only need to assume the curvature growth condition in the case when the fiber in the warped product construction is Ricci flat.
\end{remark}

In \cite{HPWLcf} we defined two natural tensors, $P$ and $Q$, which
are modifications of the Ricci tensor and the Riemann curvature tensor.
The equation in (\ref{eqnradialseccurvature}) is equivalent to the
radial flatness of $Q$, i.e., $Q(X,\nabla w,\nabla w,X)=0$.

Theorem \ref{thmradialQ} also has the same corollary as the result
in the gradient Ricci soliton case. Note in the corollary below, we
do not need the curvature growth assumption.

\begin{corollary} \label{corharmonic}
If $(M,g,w)$ is a complete
simply connected non-trivial $(\lambda,n+m)$-Einstein manifold with
harmonic curvature, then it is rigid.
\end{corollary}

\smallskip{}

The paper is organized as follows. We classify the rigid examples
in Section 2. In Section 3, we derive some identities for the tensors
$P$ and $Q$ on $(\lambda,n+m)$-Einstein manifolds with constant
scalar curvature. In Section 4 we turn our attention to proving the
classification theorems. We first prove Theorems \ref{thmuniquenesslambdazero}
and \ref{thmscaln-1}. Theorem \ref{thmdim3}, the three dimensional
classification, is then an application of these two results. In Section
5 we discuss Theorem \ref{thmradialQ} and in Section 6 we construct
the non-rigid metrics on solvable Lie groups. Further investigations
in the homogeneous case will be the subject of future paper. In Appendix
A, we prove a result for the integrability of eigen-distribution of
Ricci tensor that is used in the proof of Theorem \ref{thmradialQ},
see Theorem \ref{thmRiccidistributionintegrable}.

\medskip{}

\textbf{Acknowledgment:} The major part of this paper was done
when the first author was a visiting scholar at Department of Math,
University of Pennsylvania. He would like to thank the university
for their hospitality.

\medskip{}

\section{Rigid examples}

In this section, we give a classification of rigid $(\lambda,n+m)$-Einstein
manifolds. First recall

\begin{definition}
A $(\lambda,n+m)$-Einstein manifold $(M,g,w)$
is called \emph{rigid} if it is Einstein or its universal cover is
a product of Einstein manifolds.
\end{definition}

In the gradient Ricci soliton case, if the metric splits, then each
factor is a gradient soliton with the same value of $\lambda$ and
none of them has to be trivial. However in the $(\lambda,n+m)$-Einstein
case if the metric splits, then there are at most two Einstein constants
and one factor is trivial.

\begin{lemma} \label{lemmetricsplit}
Suppose $(M,g,w)$ is a $(\lambda,n+m)$-Einstein
metric such that the metric splits as a product
\begin{equation*}
(M,g)=(M_{1},g_{1})\times(M_{2},g_{2})
\end{equation*}
Then one of the manifolds, say $(M_{1},g_{1})$, is a trivial $(\lambda,n+m)$-Einstein
manifold.
\end{lemma}

\begin{proof}
We write the $(\lambda,n+m)$-Einstein equation in the $(1,1)$-tensor
form:
\begin{equation*}
\mathrm{Ric}-\frac{m}{w}\nabla\nabla w=\lambda I.
\end{equation*}
The operator $E\mapsto\frac{m}{w}\nabla_{E}\nabla w$ preserves the
splitting of the metric $(M,g)=(M_{1}\times M_{2},g_{1}+g_{2})$.
Using the local coordinates $x^{j}$ and writing $\nabla w=\alpha^{j}\partial_{j}$,
we have $\nabla w=X_{1}+X_{2}$ where $X_{i}$ is a vector field on
$M_{i}$, $i=1,2$. Choose a fixed point $(p,q)\in M_{1}\times M_{2}$,
then on the submanifold $\left\{ p\right\} \times M_{2}$, we have
\begin{equation*}
m\nabla X_{2}(x_{2})={w(p,x_{2})}(\mathrm{Ric}_{g_{2}}-\lambda I).
\end{equation*}
Therefore if $(M_{2},g_{2})$ is not a $\lambda$-Einstein manifold,
then $w(p,x_{2})$ does not depend on the value of $p$, i.e., $X_{1}\equiv0,$
which implies $(M_{1},g_{1})$ is a trivial $(\lambda,n+m)$-Einstein
manifold.
\end{proof}

The non-trivial $(\lambda,n+m)$-Einstein manifolds which are also Einstein were classified in \cite{CaseShuWei} and extended to manifolds with boundary in \cite{HPWLcf}. The Einstein constant which is not $\lambda$ is given by
\begin{equation}
\rho=\frac{(n-1)\lambda-\mathrm{scal}}{m-1},\label{eqnrho}
\end{equation}
and all examples are listed in the next proposition.

\begin{proposition} \label{proplambdaEinsteinEinstein}
Suppose that
$(M,g,w)$ is a non-trivial $(\lambda,n+m)$-Einstein manifold which
is also Einstein, then up to multiples of $w$ and $g$, it is isometric
to one of the examples in Table \ref{TablelambdaEinsteinEinstein}.
\end{proposition}

\begin{table}[!h]

\begin{centering}
\begin{tabular}{|c|c|c|c|c|c|}
\hline
$M$  & $g$  & $w$  & $\lambda$  & $\rho$  & $\mu$ \tabularnewline
\hline
\hline
$[-\frac{\pi}{2},\frac{\pi}{2}]$  & $\mathrm{d}r^{2}$  & $w(r)=\cos(r)$  & $m$  & $0$  & $m-1$ \tabularnewline
\hline
$[0,\infty)$  & $\mathrm{d}r^{2}$  & $w(r)=r$  & $0$  & $0$  & $m-1$ \tabularnewline
\hline
$[0,\infty)$  & $\mathrm{d}r^{2}$  & $w(r)=\sinh(r)$  & $-m$  & $0$  & $m-1$ \tabularnewline
\hline
$(-\infty,\infty)$  & $\mathrm{d}r^{2}$  & $w(r)=e^{r}$  & $-m$  & $0$  & $0$ \tabularnewline
\hline
$(-\infty,\infty)$  & $\mathrm{d}r^{2}$  & $w(r)=\cosh(r)$  & $-m$  & $0$  & $-(m-1)$ \tabularnewline
\hline
\hline
$\mathbb{D}^{n}$  & $\mathrm{d}r^{2}+\sin^{2}(r)g_{\mathbb{S}^{n-1}}$  & $w(r)=\cos(r)$  & $n+m-1$  & $n-1$  &
$m-1$ \tabularnewline
\hline
$[0,\infty)\times F$  & $\mathrm{d}r^{2}+g_{F}$  & $w(r)=r$  & $0$  & $0$  & $m-1$ \tabularnewline
\hline
$[0,\infty)\times N$  & $\mathrm{d}r^{2}+\cosh^{2}(r)g_{N}$  & $w(r)=\sinh(r)$  & $-(n+m-1)$  & $-(n-1)$  &
$m-1$ \tabularnewline
\hline
$(-\infty,\infty)\times F$  & $\mathrm{d}r^{2}+e^{2r}g_{F}$  & $w(r)=e^{r}$  & $-(n+m-1)$  & $-(n-1)$  & $0$
\tabularnewline
\hline
$\mathbb{H}^{n}$  & $\mathrm{d}r^{2}+\sinh^{2}(r)g_{\mathbb{S}^{n-1}}$  & $w(r)=\cosh(r)$  & $-(n+m-1)$  &
$-(n-1)$  & $-(m-1)$ \tabularnewline
\hline
\end{tabular}
\par
\end{centering}
\smallskip{}
\caption{Non-trivial $(\lambda,n+m)$-Einstein manifolds that are also Einstein.
Here $\mathbb{S}^{n-1}$ has Ricci curvature $n-2$, $F$ is Ricci
flat and $N$ has Ricci curvature $-(n-2)$.}
\label{TablelambdaEinsteinEinstein}
\end{table}

\begin{remark}
The first five $(\lambda,1+m)$-Einstein structures in dimension one can be viewed as degenerate cases of the last five examples. For instance, in Example 8 if $N$ is a point, then we have
Example 3.
\end{remark}

Combining these two results allows us to easily classify all rigid $(\lambda,n+m)$-Einstein manifolds.

\begin{proposition} \label{proprigidnontrivial}
A non-trivial complete rigid $(\lambda,n+m)$-Einstein manifold $(M,g,w)$ is one of the examples
in Table \ref{TablelambdaEinsteinEinstein}, or its universal cover $\widetilde{M}$ splits as
\begin{eqnarray*}
\widetilde{M} & = & (M_{1},g_{1})\times(M_{2},g_{2})\\
w & = & (c,w_{2}),
\end{eqnarray*}
where $c$ is a constant, $(M_{1},g_{1})$ is a trivial $(\lambda,n+m)$-Einstein
manifold and $(M_{2},g_{2},w_{2})$ is one of the examples in Table
\ref{TablelambdaEinsteinEinstein}.
\end{proposition}

\begin{proof}
If $M$ is Einstein, then by Proposition \ref{proplambdaEinsteinEinstein}
it is one of the examples in Table \ref{TablelambdaEinsteinEinstein}.
Otherwise the metric splits as a product of Einstein manifolds and
Lemma \ref{lemmetricsplit} implies that the potential function $w$
also splits and one of the factors, $M_{1}$ have constant potential
function and Ricci curvature $\lambda$. Now $M_{2}$ is both Einstein
and $(\lambda,n+m)$-Einstein and so applying Proposition \ref{proplambdaEinsteinEinstein}
again shows that $M_{2}$ is in Table \ref{TablelambdaEinsteinEinstein}.
\end{proof}
\medskip{}

\section{Preliminaries}

In this section we collect the formulas from \cite{KimKim}, \cite{CaseShuWei}
and \cite{HPWLcf} which we will use later in the proof of the theorems.
We apply some of these identities to give a classification of the
possible forms of $w$ on a $(\lambda,n+m)$-Einstein manifold with
$m>1$ and constant scalar curvature. Finally we derive some properties
about the critical point set of $w$.

Recall from \cite{HPWLcf} that for a $(\lambda,n+m)$-Einstein manifold
with $m\neq1$, we define
\begin{eqnarray}
\rho(x) & = & \frac{(n-1)\lambda-\mathrm{scal}}{m-1}\notag\\
P & = & \mathrm{Ric}-\rho g\notag\\
Q & = & R+\frac{2}{m}P\odot g+\frac{\rho-\lambda}{m}g\odot g\notag
\end{eqnarray}
where $R$ is the (0,4) Riemann curvature tensor, and for any two
symmetric (0,2)-tensor $s$ and $r$, $s\odot r$ is the Kulkarni-Nomizu
product defining a (0,4)-tensor
\begin{eqnarray*}
(s\odot r)(X,Y,Z,W) & = & \,\frac{1}{2}\left(r(X,W)s(Y,Z)+r(Y,Z)s(X,W)\right)\\
&  & -\frac{1}{2}\left(r(X,Z)s(Y,W)+r(Y,W)s(X,Z)\right).
\end{eqnarray*}
Note that we use the convention for the Riemann tensor that makes
$R\left(X,Y,Y,X\right)$ have the same sign as the sectional curvature
of the plane spanned by $X$ and $Y$. The Kulkarni-Nomizu tensor
is defined to be consistent with this choice.

Up to a dimensional constant, $P$ can be viewed as the Ricci tensor
associated with $Q$, i.e., if $\left\{ E_{i}\right\} _{i=1}^{n}$
is an orthonormal frame, then
\begin{equation}
\sum_{i=1}^{n}Q(X,E_{i},E_{i},Y)=\frac{n+m-2}{m}P(X,Y).\label{eqntraceQ}
\end{equation}
Also note that the trace of $P$ and $\rho$ are related by the equation
\begin{equation}
\mathrm{tr}P=(n-1)\lambda-(n+m-1)\rho.\label{trrholam}
\end{equation}

When the scalar curvature is constant, $\rho$ is constant on $M$
and $P$ is just the Ricci tensor shifted by a constant multiple of
the metric. In this case the formulas in \cite{CaseShuWei} and \cite{HPWLcf}
simplify significantly. The first set of formulas we will need involve
$P$ and the derivatives of the scalar curvature. Recall that for
a $(\lambda,n+m)$-Einstein manifold with $m>1$,
\begin{eqnarray*}
\frac{w}{2}\nabla\rho & = & P(\nabla w)\\
\frac{1}{2}\Delta(\mathrm{scal})+\frac{m+1}{2w}g\left(\nabla\mathrm{scal},\nabla w\right) & = &
(\lambda-\rho)\mathrm{tr}(P)-|P|^{2}\\
\mathrm{div}(w^{m+1}P) & = & 0.
\end{eqnarray*}
The first two equations are just formulas (3.11) and (3.12) in \cite{CaseShuWei}
rewritten in our notation, for the third identity see Proposition
5.6 in \cite{HPWLcf}. When the scalar curvature is constant these
identities give us the following formulas.

\begin{proposition} \label{propCSWconstantscal}
Let $(M,g,w)$ be
a $(\lambda,n+m)$-Einstein manifold with constant scalar curvature
and $m>1$, then
\begin{eqnarray*}
P(\nabla w) & = & 0\\
|P|^{2} & = & (\lambda-\rho)\mathrm{tr}P={\normalcolor \mathrm{constant}}.\\
\mathrm{div}(P) & = & 0.
\end{eqnarray*}
\end{proposition}

\begin{proof}
The first two are obvious from the equations above because $\mathrm{scal}$,
$\rho$, and $\mathrm{tr}P$ are all constant. For the third fact,
note that $\mathrm{div}(w^{m+1}P)=0$ is equivalent to
\begin{equation*}
\mathrm{div}(P)=-\frac{(m+1)}{w}P(\nabla w).
\end{equation*}
\end{proof}

There are two important corollaries of these formulas which we will
find useful.

\begin{corollary}[\cite{CaseShuWei}, Proposition 3.6] \label{cornrhonlambda}
Let $(M,g,w)$ be a $(\lambda,n+m)$-Einstein manifold with constant
scalar curvature and $m>1$, assume in addition that $\lambda\ne0$,
then the scalar curvature is bounded by $n\lambda$ and $n\rho$.
Moreover if $\mathrm{scal}=n\lambda$ or $n\rho$, then the manifold
is Einstein.
\end{corollary}

\begin{proof}
From the equation $|P|^{2}=(\lambda-\rho)\mathrm{tr}P$, we have
\begin{equation*}
\left|P-\frac{\mathrm{tr}P}{n}g\right|^{2}=(\mathrm{tr}P)\left(\lambda-\frac{\mathrm{scal}}{n}\right)= -\frac{1}{n}(\mathrm{scal}-n\rho)(n\lambda-\mathrm{scal}).
\end{equation*}
Then the statement follows easily.
\end{proof}

\begin{corollary} \label{kbarsign} \label{cordiffrholambda}
Let
$(M,g,w)$ be a $(\lambda,n+m)$-Einstein manifold with constant scalar
curvature and $m>1$. If $\lambda>0$($<0$), then $\lambda-\rho>0$($<0$).
\end{corollary}

\begin{proof}
Suppose that $\lambda<0$ and $\rho\leq\lambda$, then by (\ref{trrholam}),
$\mathrm{tr}(P)>0$. Then the formula
\begin{equation*}
|P|^{2}=(\lambda-\rho)\mathrm{tr}P\leq0
\end{equation*}
shows that $M$ is Einstein and $\rho=\lambda$. However, this case
never occurs in Table \ref{TablelambdaEinsteinEinstein} with $m>1$
so we have a contradiction.

A similar argument shows $\rho-\lambda<0$ if $\lambda>0$. In fact,
constant scalar curvature is not necessary for the conclusion in the
$\lambda>0$ case, see Proposition 5.4 in \cite{HPWLcf}.
\end{proof}

The other set of formulas we will use involve the covariant derivatives
of the tensors $P$ and $Q$, and are proven as Proposition 6.2 in
\cite{HPWLcf}.

\begin{equation}
\frac{w}{m}\left((\nabla_{X}P)(Y,Z)-(\nabla_{Y}P)(X,Z)\right)=-Q(X,Y,Z,\nabla w)-\frac{1}{m}(g\odot
g)\left(X,Y,Z,P(\nabla w)\right)\label{symderivP}
\end{equation}

In the case of constant scalar curvature these give us the following
identities.

\begin{proposition} \label{propderivPQ}
Suppose $(M,g,w)$ is a
$(\lambda,n+m)$-Einstein manifold with constant scalar curvature
and $m>1$, then
\begin{eqnarray*}
\frac{w}{m}\left((\nabla_{X}P)(Y,Z)-(\nabla_{Y}P)(X,Z)\right) & = &
\frac{w}{m}\left((\nabla_{X}\mathrm{Ric})(Y,Z)-(\nabla_{Y}\mathrm{Ric})(X,Z)\right)=-Q(X,Y,Z,\nabla w),\\
\frac{w}{m}(\nabla_{\nabla w}P)(X,Y) & = &
-\left(\frac{w}{m}\right)^{2}(\lambda-\rho)P(X,Y)+\left(\frac{w}{m}\right)^{2}g(P(X),P(Y))\\
&  & +Q(\nabla w,X,Y,\nabla w).
\end{eqnarray*}
\end{proposition}

\begin{proof}
The first equation comes from combining the equation (\ref{symderivP})
with the fact that $P(\nabla w)=0$ and $\rho$ is constant. The second
equation then follows from combining the first equation with the following
formula
\begin{eqnarray*}
(\nabla_{X}P)(\nabla w,Y) & = & P(\nabla_{X}\nabla w,Y)\\
& = & \frac{w}{m}P\left((\mathrm{Ric}-\lambda I)(X),Y\right)\\
& = & \frac{w}{m}P\left((P+(\rho-\lambda)I)(X),Y\right),
\end{eqnarray*}
where we have assumed that $X$ and $Y$ are arbitrary parallel fields.
\end{proof}

We obtain the following corollary from considering the last identity
at a critical point of $w$.

\begin{corollary} \label{corPidentitynablawzero}
If $\nabla w$
vanishes at $p\in M$, then
\begin{equation*}
P\circ(P-(\lambda-\rho)I)=0.
\end{equation*}
\end{corollary}

\smallskip{}

The other very important formula we will utilize is the identity from
\cite{KimKim} which states that
\begin{equation}
w\Delta w+(m-1)|\nabla w|^{2}+\lambda w^{2}=\mu=\mathrm{const}.\label{eqnmu}
\end{equation}

By tracing the $(\lambda,n+m)$-Einstein equation, we also have
\begin{equation*}
\Delta w=\frac{w}{m}\left(\mathrm{scal}-n\lambda\right).
\end{equation*}
Letting $\bar{\mu}=\frac{\mu}{m-1}$ after a calculation the equation
(\ref{eqnmu}) then becomes
\begin{equation}
\bar{\mu}=\bar{k}w^{2}+|\nabla w|^{2},\quad\mbox{where}\quad\bar{k}=\frac{\lambda-\rho}{m}.\label{eqnmubar}
\end{equation}
Since $\bar{k}$ is constant when the scalar curvature is constant
this tells us that the only possibilities for the form of the function
$w$ are the functions appearing in the rigid examples.

\begin{proposition} \label{propwfunctions}
Suppose $(M,g,w)$ is
a non-trivial $(\lambda,n+m)$-Einstein manifold with constant scalar
curvature and $m>1$.
\begin{itemize}
\item If $\bar{k}>0$, then $\bar{\mu}>0$ and
\begin{equation*}
w=\sqrt{\frac{\bar{\mu}}{\bar{k}}}\cos\left(\sqrt{\bar{k}}r\right),
\end{equation*}
for a distance function $r.$
\item If $\bar{k}=0$, then $\bar{\mu}>0$ and $w=\sqrt{\bar{\mu}}r$ for
a distance function $r.$
\item If $\bar{k}<0$, then the general form of $w$ is
\begin{equation*}
w=C_{1}\exp\left({\sqrt{-\bar{k}}r}\right)+C_{2}\exp\left({-\sqrt{-\bar{k}}r}\right)
\end{equation*}
where $r$ is a distance function. More specifically we can express
this as
\begin{eqnarray*}
w & = & \exp\left({\sqrt{-\bar{k}}r}\right),\text{ when }\bar{\mu}=0,\\
w & = & \sqrt{\frac{\bar{\mu}}{\bar{k}}}\cosh\left(\sqrt{-\bar{k}}r\right),\text{ when }\bar{\mu}<0,\\
w & = & \sqrt{-\frac{\bar{\mu}}{\bar{k}}}\sinh\left(\sqrt{-\bar{k}}r\right),\text{ when }\bar{\mu}>0.
\end{eqnarray*}
\end{itemize}
\end{proposition}

\begin{remark}
Note that Corollary \ref{kbarsign} shows that $\bar{k}$
and $\lambda$ always have the same sign.
\end{remark}

\begin{proof}
Recall that a distance function is simply a smooth function whose
gradient always has unit length. The proof uses the equation
\begin{equation*}
\bar{k}w^{2}+\left\vert \nabla w\right\vert ^{2}=\bar{\mu}.
\end{equation*}

When $\bar{k}=0$ and $\bar{\mu}>0$ we note that $\frac{w}{\sqrt{\bar{\mu}}}$
is a distance function.

When $\bar{\mu}=0$ and $\bar{k}<0$ the function $\log\frac{w}{\sqrt{-\bar{k}}}$
is a distance function.

When $\bar{k}>0$ we note that $\bar{\mu}>0$ and we can rewrite the
formula as
\begin{equation*}
\frac{1}{\bar{k}}\frac{\left\vert \nabla w\right\vert
^{2}}{\left(\sqrt{\frac{\bar{\mu}}{\bar{k}}}\right)^{2}-w^{2}}=1
\end{equation*}
showing that $\sqrt{\bar{k}}\arccos\left(\frac{w}{\sqrt{\frac{\bar{\mu}}{\bar{k}}}}\right)$
is a distance function.

Finally when $\bar{k}<0$ there are two cases depending on the sign
of $\bar{\mu}.$ The formula is then rewritten as
\begin{equation*}
\frac{1}{-\bar{k}}\frac{\left\vert \nabla w\right\vert
^{2}}{\left(\sqrt{-\frac{\bar{\mu}}{\bar{k}}}\right)^{2}+w^{2}}=1
\end{equation*}
when $\bar{\mu}>0$, or
\begin{equation*}
\frac{1}{\bar{k}}\frac{\left\vert \nabla w\right\vert
^{2}}{\left(\sqrt{\frac{\bar{\mu}}{\bar{k}}}\right)^{2}-w^{2}}=1
\end{equation*}
when $\bar{\mu}<0.$ And we get the specific expression using $\mathrm{arcsinh}$
or $\mathrm{arccosh}$.
\end{proof}

Finally in this section we discuss two propositions about the critical
point set of $w$ that will be used in the proof of Theorem \ref{thmradialQ}.

\begin{proposition} \label{propNequaldim}
Suppose $(M,g,w)$ is
a $(\lambda,n+m)$-Einstein manifold with constant scalar curvature
and the set of critical points of $w$ is non-empty. Then all connected
components have the same dimension. Furthermore let $N$ be a connected
component, then normal vectors to $N$ are $0$ eigenvectors for $P$
and tangent vectors are $\left(\lambda-\rho\right)$ eigenvectors
for $P$.
\end{proposition}

\begin{proof}
There are only two nontrivial cases where $w$ has critical points
\begin{eqnarray*}
w & = & C\cos\left(\sqrt{\bar{k}}r\right),\text{ or}\\
w & = & C\cosh\left(\sqrt{-\bar{k}}r\right).
\end{eqnarray*}
By scaling we can further assume that $C=1$ and additionally that
$N^{l}\subset\left\{ r=0\right\} $.

We know that at $N$
\begin{equation*}
\Delta r=\frac{n-l-1}{r}+O\left(1\right).
\end{equation*}
Thus
\begin{eqnarray*}
\Delta w & = & \left(-\bar{k}r+O\left(r^{3}\right)\right)\left(\frac{n-l-1}{r} +
O\left(1\right)\right)+\left(-\bar{k}+O\left(r^{2}\right)\right) \\
& = & -\bar{k}\left(n-l\right)+O\left(r\right).
\end{eqnarray*}

The $\left(\lambda,n+m\right)$-Einstein equation gives us
\begin{eqnarray*}
\Delta w & = & \frac{w}{m}\left(\mathrm{scal}-n\lambda\right)\\
& = & \frac{w}{m}\left(\mathrm{tr}P-n\left(\lambda-\rho\right)\right).
\end{eqnarray*}
So at $N$ we have
\begin{equation*}
\frac{1}{m}\left(\mathrm{tr}P-n\left(\lambda-\rho\right)\right)=-\bar{k}\left(n-l\right)
\end{equation*}
showing that
\begin{eqnarray*}
\mathrm{tr}P & = & -\left(\lambda-\rho\right)\left(n-l\right)+n\left(\lambda-\rho\right)\\
& = & l\left(\lambda-\rho\right).
\end{eqnarray*}
Corollary \ref{corPidentitynablawzero} implies that at $N$
\begin{equation*}
P\circ\left(P-(\lambda-\rho)I\right)=0.
\end{equation*}
Thus $0$ and $\lambda-\rho$ are the only possible eigenvectors
at $N.$ Since $\frac{\nabla w}{\left\vert \nabla w\right\vert }$
converges to normal vectors to $N$ it follows that any normal vector
to $N$ is a $0$ eigenvector for $P.$ As
\begin{equation*}
\mathrm{tr}P=l\left(\lambda-\rho\right)
\end{equation*}
the tangent vectors to $N$ must be $\lambda-\rho$ eigenvectors
for $P.$
\end{proof}

\begin{proposition} \label{propcriticalpoints}
Suppose a $(\lambda,n+m)$-Einstein
manifold $(M,g,w)$ satisfies the following identity
\begin{equation*}
P\circ\left(P-(\lambda-\rho)I\right)=0
\end{equation*}
everywhere on $M$, then for $N$, the set of critical points of
$w$, we have
\begin{enumerate}
\item $N$ is totally geodesic.
\item $\nabla P$ vanishes at $N.$
\end{enumerate}
\end{proposition}

\begin{proof}
The equation
\begin{equation*}
m\nabla\nabla w=w\left(P-\left(\lambda-\rho\right)I\right)
\end{equation*}
at $N$ reduces to a soliton type equation
\begin{equation*}
m\nabla\nabla w=P-\left(\lambda-\rho\right)I.
\end{equation*}
In addition we also have from Proposition \ref{propderivPQ} that
at $N$
\begin{equation*}
0=-Q(X,Y,Z,\nabla w)=\frac{w}{m}\left((\nabla_{X}P)(Y,Z)-(\nabla_{Y}P)(X,Z)\right).
\end{equation*}
First we show that $N$ is totally geodesic. Let $Y$ be a normal
vector field to $N$ and $X$ be a tangent vector field. Thus
\begin{eqnarray*}
P(Y)=0 & \qquad & \nabla_{Y}\nabla w=\bar{k}Y\\
P(X)=\left(\lambda-\rho\right)X & \qquad & \nabla_{X}\nabla w=0.
\end{eqnarray*}

Using that the only eigenvalues for $P$ are $0$ and $\lambda-\rho$
we can extend $X,Y$ such that they remain eigenfields for $P$. In
particular
\begin{equation*}
-P\left(\nabla_{X}Y\right)=\left(\nabla_{X}P\right)(Y)=\left(\nabla_{Y}P\right)(X)=
\left(\lambda-\rho\right)\nabla_{Y}X-P\left(\nabla_{Y}X\right).
\end{equation*}

Then from the soliton equation we see that
\begin{eqnarray*}
\left(\lambda-\rho\right)g\left(\nabla_{X}Y,X\right) & = &
P\left(\nabla_{X}Y,X\right)-m\mathrm{Hess}w\left(\nabla_{X}Y,X\right)\\
& = & P\left(\nabla_{X}Y,X\right)\\
& = & P\left(\nabla_{Y}X,X\right)-\lambda g\left(\nabla_{Y}X,X\right)+\rho g\left(\nabla_{Y}X,X\right) \\
& = & m\mathrm{Hess}w\left(\nabla_{Y}X,X\right)\\
& = & 0.
\end{eqnarray*}
So the second fundamental form vanishes and $N$ is totally geodesic.

Next we show that $P$ is parallel at $N.$ We show that
\begin{eqnarray*}
\nabla_{X}P & = & 0,\\
\nabla_{Y}P & = & 0,
\end{eqnarray*}
where $X$ is tangent to $N,$ and $Y$ is normal to $N.$

To show the first we evaluate on $X^{\prime}\in TN$ and $Y^{\prime}$
normal to $N.$ Since
\begin{eqnarray*}
P(X^{\prime}) & = & \left(\lambda-\rho\right)X^{\prime}\\
P(Y^{\prime}) & = & 0
\end{eqnarray*}
we obtain
\begin{eqnarray*}
\left(\nabla_{X}P\right)\left(X^{\prime}\right) & = &
\left(\lambda-\rho\right)\nabla_{X}X^{\prime}-P\left(\nabla_{X}X^{\prime}\right),\\
\left(\nabla_{X}P\right)\left(Y^{\prime}\right) & = & -P\left(\nabla_{X}Y^{\prime}\right).
\end{eqnarray*}
Both of these expressions vanish as $N$ is totally geodesic.

For the second case use Proposition \ref{propderivPQ} again to obtain:
\begin{eqnarray*}
\frac{w}{m}(\nabla_{\nabla w}P)(Z) & = &
-\left(\frac{w}{m}\right)^{2}\left(\left((\lambda-\rho)I-P\right)P\right)\left(Z\right)+Q\left(Z,\nabla
w\right)\nabla w\\
& = & Q\left(Z,\nabla w\right)\nabla w.
\end{eqnarray*}
Dividing by $\left\vert \nabla w\right\vert $ then yields
\begin{equation*}
\frac{w}{m}(\nabla_{\frac{\nabla w}{\left\vert \nabla w\right\vert }}P)(Z)=Q\left(Z,\nabla w\right)\frac{\nabla
w}{\left\vert \nabla w\right\vert }
\end{equation*}
and supposing that
\begin{equation*}
\frac{\nabla w}{\left\vert \nabla w\right\vert }\rightarrow Y
\end{equation*}
as we approach $N$ we obtain
\begin{equation*}
-\frac{w}{m}\nabla_{Y}P=0
\end{equation*}
as desired.
\end{proof}

\begin{remark}
If a $(\lambda,n+m)$-Einstein manifold satisfies
the radial $Q$ flatness condition (\ref{eqnradialseccurvature}),
then we will show that the $P$ tensor satisfies the equation in Proposition
\ref{propcriticalpoints}.
\end{remark}

\medskip{}

\section{Proof of Theorem \protect\ref{thmdim3}, \protect\ref{thmuniquenesslambdazero},
and \protect\ref{thmscaln-1}}

In this section we will discuss the proofs of theorems \ref{thmdim3},
\ref{thmuniquenesslambdazero}, and \ref{thmscaln-1}. The easiest
to prove is Theorem \ref{thmuniquenesslambdazero} which is the classification
in the $\lambda=0$ case. It already follows directly from the formulas
in the past section, and is the same argument as in \cite{CaseShuWei}.

\begin{proof}[Proof of Theorem \protect\ref{thmuniquenesslambdazero}]
Since the scalar curvature is constant and $\lambda=0$, from Proposition
\ref{propCSWconstantscal} we have
\begin{eqnarray*}
\rho & = & -\frac{1}{n+m-1}\mathrm{tr}(P),\\
|P|^{2} & = & -\rho\mathrm{tr}(P)=\frac{1}{n+m-1}\left(\mathrm{tr}P\right)^{2}.
\end{eqnarray*}
The second identity above and the Cauchy-Schwarz inequality $n|P|^{2}\geq\left(\mathrm{tr}(P)\right)^{2}$
imply that
\begin{equation*}
\frac{n}{n+m-1}\left(\mathrm{tr}(P)\right)^{2}\geq\left(\mathrm{tr}(P)\right)^{2}.
\end{equation*}
So we have $\mathrm{tr}(P)=0$ as $m>1$. It follows that $\rho=0$
and then $|P|^{2}=0$. Hence $\mathrm{Ric}=\rho g$ and the result
follows from the classification in Proposition \ref{proplambdaEinsteinEinstein}.
\end{proof}
\medskip{}

Next we turn our attention to Theorem \ref{thmscaln-1}. Before the
proof we give two corollaries that follow from combining the theorem
with some of the other results we have already discussed.

\begin{corollary} \label{corrhozero}
Let $(M,g,w)$ be a $(\lambda,n+m)$-Einstein
manifold with constant scalar curvature and $m>1$. If $\lambda>0$,
then
\begin{equation*}
0\leq\rho\leq\lambda,\end{equation*}
and if $\lambda<0$, then
\begin{equation*}
0\geq\rho\geq\lambda.\end{equation*}
Moreover, in either case, $\rho=\lambda$ if and only if $\lambda=0$
and the metric is rigid, and $\rho=0$ if and only if $M=N\times\mathbb{R}$
where $N$ is a $\lambda$-Einstein metric. \end{corollary}
\begin{proof}
In Corollary \ref{kbarsign} we already saw that $\lambda\geq\rho$
when $\lambda>0$ and that $\lambda\leq\rho$ when $\lambda<0$. To
see the rigidity statement for this side of inequality, note that
if $\lambda=\rho$, then $\bar{k}=0$ and then we get rigidity by
combining Proposition \ref{propwfunctions} and Theorem \ref{thmuniquenesslambdazero}.

The other inequality is equivalent to Theorem \ref{thmscaln-1}. To
see this note that, since $\mathrm{Ric}(\nabla w)=\rho\nabla w$,
the hypothesis on $\mathrm{Ric}(\nabla w,\nabla w)$ is equivalent
to assuming $\rho$ is zero or has the opposite sign of $\lambda$.
\end{proof}

As mentioned in the introduction, Theorem \ref{thmscaln-1} can also
be interpreted as a gap theorem about the scalar curvature.

\begin{corollary} Let $(M,g,w)$ be a $(\lambda,n+m)$-Einstein manifold
with constant scalar curvature and $m>1$. If the scalar curvature
is between $(n-1)\lambda$ and $n\lambda$ then either the metric
is a $\lambda$-Einstein metric, or it is rigid and splits as $N\times\mathbb{R}$
where $N$ is a $\lambda$-Einstein metric. \end{corollary}

\begin{proof}
As we have seen in Proposition \ref{cornrhonlambda}, the scalar curvature
must be between $n\rho$ and $n\lambda$, and can only be equal to
$n\lambda$ if the metric is $\lambda$-Einstein. On the other hand,
by the definition of $\rho$, the scalar curvature being bounded away
from zero by $(n-1)\lambda$ is equivalent to $\rho$ having the opposite
sign as $\lambda$, so the other half of the result is equivalent
to Theorem \ref{thmscaln-1}.
\end{proof}

Now we prove the theorem.

\begin{proof}[Proof of Theorem \protect\ref{thmscaln-1}]
For the discussion above we see that the hypothesis is equivalent to $\rho\bar{k}\leq0$.

By Proposition \ref{propwfunctions} we see that if $w$ is not constant,
then $w=w(r)$ where $r$ is a distance function and $w^{\prime\prime}=-\bar{k}w.$
This implies
\begin{eqnarray*}
\nabla w & = & w^{\prime}\nabla r\\
\Delta w & = & -\bar{k}w+w^{\prime}\Delta r\end{eqnarray*}
and
\begin{equation*}
\Delta w=\frac{w}{m}\left(\mathrm{scal}-n\lambda\right)
\end{equation*}
so
\begin{equation*}
-\bar{k}w+w^{\prime}\Delta r=\frac{w}{m}\left(\mathrm{scal}-n\lambda\right)
\end{equation*}
which implies that
\begin{equation*}
w^{\prime}\Delta r=w\left(\frac{\mathrm{scal}-n\lambda}{m}+\frac{\lambda-\rho}{m}\right)=-\rho w.
\end{equation*}
The Bochner formula for $r$ is
\begin{eqnarray*}
0 & = & \left\vert \mathrm{Hess}r\right\vert ^{2}+g\left(\nabla r,\nabla\Delta r\right)+\mathrm{Ric}\left(\nabla
r,\nabla r\right)\\
& = & \left\vert \mathrm{Hess}r\right\vert ^{2}+g\left(\nabla r,\nabla\Delta r\right)+\rho.
\end{eqnarray*}
The middle term can be calculated by using
\begin{eqnarray*}
-\rho w^{\prime}\nabla r & = & -\bar{k}w\Delta r\nabla r+w^{\prime}\nabla\Delta r\\
& = & \frac{-\bar{k}w}{w^{\prime}}\left(-\rho w\right)\nabla r+w^{\prime}\nabla\Delta r\\
& = & \rho\frac{\bar{k}w^{2}}{w^{\prime}}\nabla r+w^{\prime}\nabla\Delta r
\end{eqnarray*}
showing that
\begin{equation*}
g\left(\nabla r,\nabla\Delta r\right)=-\rho+\rho\bar{k}\frac{w^{2}}{({w^{\prime}})^{2}}.
\end{equation*}
Thus we obtain
\begin{eqnarray*}
0 & = & \left\vert \mathrm{Hess}r\right\vert ^{2}+g\left(\nabla r,\nabla\Delta r\right)+\rho\\
& = & \left\vert \mathrm{Hess}r\right\vert ^{2}-\rho\bar{k}\frac{w^{2}}{\left(w^{\prime}\right)^{2}}
\end{eqnarray*}
showing that $\rho\bar{k}\geq0,$ and can only vanish when $\mathrm{Hess}r=0$,
which implies the splitting along the gradient of $w$.
\end{proof}

\medskip{}

We finish this section by showing Theorem \ref{thmdim3}, i.e., a
three dimensional $(\lambda,3+m)$-Einstein manifold is rigid if it
has constant scalar curvature.

\begin{proof}[Proof of Theorem \protect\ref{thmdim3}]
We start by showing that $P\ $has constant eigenvalues. Proposition \ref{propCSWconstantscal}
says that
\begin{eqnarray*}
P\left(\nabla w\right) & = & 0,\\
\mathrm{tr}P & = & 2\lambda-(m+2)\rho,\\
|P|^{2} & = & (\lambda-\rho)\mathrm{tr}P.
\end{eqnarray*}
Thus one eigenvalue is 0 and if the other two are $p_{1}$ and $p_{2}$ then
\begin{equation*}
p_{1,2}=\frac{1}{2}\left(2\lambda-(m+2)\rho\pm\sqrt{2m\rho\lambda-2m\rho^{2}-m^{2}\rho^{2}}\right) =\frac{1}{2}\left(\mathrm{tr}P\pm\sqrt{m\rho(\mathrm{tr}P)}\right).
\end{equation*}
This shows in particular that $\left(\nabla P\right)\left(E,E\right)=0$
if $E$ is a unit eigenvector for $P.$

The goal now is to prove that either $\lambda=0$ in which case we
can use Theorem \ref{thmuniquenesslambdazero} or $p_{1}=p_{2}$.
In the latter case the metric is either $\rho$-Einstein or $\rho=0$
reducing us respectively to Proposition \ref{proplambdaEinsteinEinstein}
or Corollary \ref{corrhozero}.

Recall
\begin{equation*}
Q=R+\frac{2}{m}P\odot g+\frac{\rho-\lambda}{m}g\odot g.
\end{equation*}
If $n=3$, then we have $R=\frac{\mathrm{scal}}{6}g\odot g+(\mathrm{Ric}-\frac{2\mathrm{scal}}{3}g)\odot g$.
Since $\mathrm{scal}=\mathrm{tr}P+3\rho$, the tensor $Q$ can be
written as
\begin{equation*}
Q=\frac{m+1}{m}\left(2P\odot g-\frac{1}{2}(\mathrm{tr}P)g\odot g\right).
\end{equation*}
If $E_{i}$ is a unit eigenfield for $p_{i},$ then
\begin{eqnarray*}
Q(\nabla w,E_{i},E_{i},\nabla w) & = & \frac{m+1}{m}\left(|\nabla
w|^{2}P(E_{i},E_{i})-\frac{\mathrm{tr}P}{2}|\nabla w|^{2}g(E_{i},E_{i})\right)\\
& = & \frac{m+1}{2m}(p_{i}-p_{j})|\nabla w|^{2}
\end{eqnarray*}
where $j\neq i.$

Using Proposition \ref{propderivPQ} we have
\begin{eqnarray*}
\frac{w}{m}(\nabla_{\nabla w}P)(X,Y) & = &
-\left(\frac{w}{m}\right)^{2}(\lambda-\rho)P(X,Y)+\left(\frac{w}{m}\right)^{2}g(P(X),P(Y))\\
&  & +Q(\nabla w,X,Y,\nabla w).
\end{eqnarray*}
Evaluating on $E_{i}$ yields
\begin{equation*}
0=-\left(\frac{w}{m}\right)^{2}(\lambda-\rho)p_{i}+\left(\frac{w}{m}\right)^{2}p_{i}^{2}
+\frac{m+1}{2m}(p_{i}-p_{j})|\nabla w|^{2}.
\end{equation*}

When $M$ has boundary we know that $w=0$ somewhere so this formula
immediately shows that $p_{1}=p_{2}$. In general we can subtract
the two equations to obtain
\begin{equation*}
0=\left(p_{1}-p_{2}\right)\left(-\frac{w^{2}}{m}(\lambda-\rho)+
\frac{w^{2}}{m}\left(p_{1}+p_{2}\right)+\left(m+1\right)\left\vert \nabla w\right\vert ^{2}\right).
\end{equation*}
This shows that either $p_{1}=p_{2}$ or
\begin{equation*}
2\lambda-(m+2)\rho=\mathrm{tr}P=\left(\lambda-\rho\right)-m\left(m+1\right)\frac{\left\vert \nabla w\right\vert
^{2}}{w^{2}}.
\end{equation*}
As
\begin{equation*}
\frac{\lambda-\rho}{m}w^{2}+\left\vert \nabla w\right\vert ^{2}=\bar{\mu}
\end{equation*}
the latter case can only happen when $\bar{\mu}=0$ and
\begin{equation*}
\frac{-\left(\lambda-\rho\right)+m\rho}{m+1}=m\frac{\left\vert \nabla w\right\vert
^{2}}{w^{2}}=-\left(\lambda-\rho\right).
\end{equation*}
Thus $\lambda=0$.
\end{proof}

\begin{remark}
In the proof above, we showed that $\mathrm{Ric}$
has constant eigenvalues. Such a metric is called \emph{Ricci curvature
homogeneous}. This condition is more general than being \emph{curvature
homogeneous}, i.e., for any two points $p$ and $q$ in $M$, there
exists a linear isometry $\phi:T_{p}M\rightarrow T_{q}M$ such that
$\phi^{\ast}\left(R_{q}\right)=R_{p}$, see \cite{Singer}. Both notions
are equivalent in the two and three-dimensional cases but not for
higher dimensions. In dimension three there are curvature homogeneous
spaces that are not homogeneous, see \cite{Bueken,BuekenVanhecke,Kowalski}
and the references therein.
\end{remark}

\medskip{}

\section{Proof of Theorem \protect\ref{thmradialQ} and Corollary \protect\ref{corharmonic}}

In this section we prove Theorem \ref{thmradialQ}, i.e., that radial flatness of $Q$ implies the rigidity. There are a number of steps and the proof breaks down into different cases. The proof comes from studying the eigenvalues and eigen-distributions of $P$. First we show that $0$ and $\lambda-\rho$ are the only possible eigenvalues.

\begin{lemma} \label{lemradialQtwoeigenvalues}
Suppose a $(\lambda,n+m)$-Einstein manifold $(M,g,w)$ has constant scalar curvature and
\begin{equation*}
Q(\nabla w,\cdot,\cdot,\nabla w)=0.
\end{equation*}
Then the eigenvalues of $P$ are either $0$ or $\lambda-\rho$, i.e.,
on $M$ we have
\begin{equation*}
P\circ\left(P-(\lambda-\rho)I\right)=0.
\end{equation*}
\end{lemma}
\begin{proof}
Since the scalar curvature is constant, from Proposition \ref{propderivPQ} we have
\begin{eqnarray*}
(\nabla_{\nabla w}P) & = & \frac{w}{m}P\circ(P-(\lambda-\rho)I)\\
& = & \frac{w}{m}P\circ(P-m\bar{k}I),
\end{eqnarray*}
where $\bar{k}=\frac{\lambda-\rho}{m}\ne0$ by Corollary \ref{cordiffrholambda}.

In addition from Proposition \ref{propwfunctions} if $\bar{k}>0$ then
\begin{equation*}
w=\cos\left(\sqrt{\bar{k}}r\right),
\end{equation*}
and if $\bar{k}<0$
\begin{eqnarray*}
w & = & \exp\left({\sqrt{-\bar{k}}r}\right),\\
w & = & \cosh\left(\sqrt{-\bar{k}}r\right),\quad\text{ or}\\
w & = & \sinh\left(\sqrt{-\bar{k}}r\right).
\end{eqnarray*}
At points where $\nabla w\neq0$ the radial curvature equation can be rewritten as
\begin{equation*}
(\nabla_{\nabla r}P)=\frac{w}{mw^{\prime}}\left(P\circ(P-m\bar{k}I)\right).
\end{equation*}
Let $Y$ be a unit parallel vector field along an integral curve of $r$ and let $y(r)=P(Y,Y)$. So the above equation becomes
\begin{equation*}
y^{\prime}(r)=\frac{w}{mw^{\prime}}y(y-m\bar{k}).
\end{equation*}
Using that $w^{\prime\prime}=-\bar{k}w$ we see that this equation has the general nonzero solution
\begin{eqnarray*}
y\left(r\right) & = & m\bar{k}\frac{w^{\prime}}{w^{\prime}+A}\\
& = & \frac{m\bar{k}}{1+A\left(w^{\prime}\right)^{-1}}.
\end{eqnarray*}
Choose an orthonormal frame $\left\{ Y_{i}\right\} _{1}^{n-1}$ in the normal space of $\nabla r$ and assume that they are parallel vector fields along $\nabla r$. Suppose $y_{i}=P(Y_{i},Y_{i})$,
$i=1,2,\dots,p$, are the non-constant values, i.e.,
\begin{equation*}
y_{i}=\frac{m\bar{k}}{1+A_{i}z(r)},
\end{equation*}
where $z(r)=\left(w^{\prime}\right)^{-1}$ and $A_{i}$'s are nonzero constants. As $\mathrm{tr}(P)$ is constant, it follows that
\begin{equation*}
\sum_{i=1}^{p}y_{i}=0.
\end{equation*}
To see the constant on the right hand side is zero, let $r\rightarrow0$. Now plugging in the specific form for $y$ we obtain
\begin{equation*}
\frac{p+b_{1}z+b_{2}z^{2}+\dots+b_{p}z^{p}}{\Pi_{i=1}^{p}(1+A_{i}z)}=0,
\end{equation*}
where the coefficients $b_{i}$ can be derived from the binomial formula. However, unless $p=0$ this is a contradiction since $z(r)$ has the property that $\left\{ 1,z,z^{2},z^{3},\dots\right\} $ is a linearly independent set of functions on their intervals of definition.

Therefore, in all cases, $P(Y_{i},Y_{i})$ is constant and thus either $0$ or $\lambda-\rho$. It follows that the eigenvalues of $P$ are bounded by $0$ and $\lambda-\rho$. Using the identity $\vert P\vert^{2}=(\lambda-\rho)\mathrm{tr}P$ in Proposition \ref{propCSWconstantscal}, it is either $0$ or $\lambda-\rho$.
\end{proof}

From Proposition \ref{propwfunctions}, there are four different cases:
\begin{enumerate}
\item $w=\cos\left(\sqrt{\bar{k}}r\right)$, $M$ has boundary and $\bar{\mu}>0$,
\label{casewcos}
\item $w=\cosh\left(\sqrt{-\bar{k}}r\right)$, $M$ has no boundary and
$\bar{\mu}<0$, \label{casewcosh}
\item $w=\sinh\left(\sqrt{-\bar{k}}r\right)$, $M$ has boundary and $\bar{\mu}>0$,
\label{casewsinh}
\item $w=\exp\left(\sqrt{-\bar{k}}r\right)$, $M$ has no boundary and $\bar{\mu}=0$.
\label{casewexp}
\end{enumerate}
In the first two cases, the critical point set of $w$ is non-empty and the proof follows from a similar argument as in Ricci soliton case, see \cite{PWclassification}.

\begin{proof}[Proof of Theorem \ref{thmradialQ} in Case (\ref{casewcos}) and (\ref{casewcosh})]
This comes from considering the critical point set of $w$, $N=\set{x:\nabla w(x)=0}$. From Proposition \ref{propcriticalpoints} we see that $N$ is a totally geodesic $\lambda$-Einstein manifold.

The normal exponential map $\nu\left(N\right)\rightarrow M$ follows the integral curves for $\nabla w$(or $\nabla r$) and therefore is a diffeomorphism. Using the fundamental equations, see for example \cite[Chapter 2.5]{PetersenGTM}, the metric on $M$ is completely determined by the radial sectional curvatures and the metric on $N$ as $N \subset M$ is totally geodesic. Since these match exactly with the values in the corresponding rigid cases, the metric must be rigid.
\end{proof}

Now we are left with Case (\ref{casewsinh}) and (\ref{casewexp}). The proof of the theorem in these cases is much more involved since $N=\emptyset$.

From the explicit formula of the function $w$ the maximal interval $I$ of $r$ is either $(-\infty,\infty)$ or $[0,\infty)$. On this interval we can write the metric as
\begin{eqnarray*}
M & = & I\times\Sigma\\
g & = & \mathrm{d}r^{2}+g_{r}.
\end{eqnarray*}
From Lemma \ref{lemradialQtwoeigenvalues}, the tangent space of $\Sigma$ has an orthogonal splitting of eigen-distributions of $P$:
\begin{eqnarray*}
T\Sigma & = & \mathcal{N}\oplus\mathcal{P}\\
P|_{\mathcal{N}} & = & (\lambda-\rho)\mathrm{id}_{\mathcal{N}}\\
P|_{\mathcal{P}} & = & 0,
\end{eqnarray*}
and they are parallel along $\nabla r$.  

Theorem \ref{thmradialQ} will follow by showing that the distributions
are parallel, see Lemma \ref{lemparallelmupositive} for Case (\ref{casewsinh})
and Lemma \ref{lemparallelmuzero} for Case (\ref{casewexp}). In
Lemma \ref{lemparallelmuzero} we assume that the curvature does not
grow as fast as exponentially. In fact we can show that the eigen-distribution
of $P$ is integrable without this curvature growth assumption, see
Proposition \ref{propriccihypersurface} and Theorem \ref{thmRiccidistributionintegrable}
in Appendix A.

Before proceeding to the general case we note that these steps are considerably simpler in the harmonic curvature case. In fact, the argument in \cite[Proposition 16.11]{BesseEinstein} shows that the eigen-distributions of $P$ are always integrable in this case. Moreover, if the curvature is harmonic, then
\begin{equation*}
\left(\nabla_{E_{1}}\mathrm{Ric}\right)(E_{2},E_{3})-\left(\nabla_{E_{2}}\mathrm{Ric}\right)(E_{1},E_{3})= \mathrm{div}R(E_{1},E_{2},E_{3})=0
\end{equation*}
for any vector fields $E_{1}$, $E_{2}$ and $E_{3}$. The scalar curvature is also constant and thus Proposition \ref{propderivPQ} shows that $Q(E_{1},E_{2},E_{3},\nabla w)=0$. This explains why harmonic curvature is a stronger assumption than $Q(\nabla w,E,F,\nabla w)=0$. For any two eigenvector fields $X$ with $P(X)=(\lambda-\rho)(X)$ and $U$ with $P(U)=0$ the vanishing of $Q$ also implies that
\begin{equation*}
R(X,U,X,\nabla r)=R(U,X,U,\nabla r)=0.
\end{equation*}

On the other hand, we also have the following calculation. We will only use the first and third equations for the harmonic curvature case, but will find the other two equations useful later.

\begin{proposition} \label{propRadialCurvatureOperator} 
Suppose $X,Y$ and $U,V$ are $\lambda-\rho$ and 0 eigenvector fields of $P$ respectively. Then we have 
\begin{eqnarray}
R(X,U,Y,\nabla r) & = & -\bar{k}\frac{w}{w^{\prime}}g\left(\nabla_{X}Y,U\right)\label{eqnRXUX}\\
\left(\nabla_{\nabla r}R\right)(X,U,Y,\nabla r) & = & \frac{w^{\prime}}{w}R(X,U,Y,\nabla r)\label{eqnevolutionRXUX}\\
&  & +\bar{k}\frac{w}{w^{\prime}}\left(R(X,U,Y,\nabla r)+R(Y,U,X,\nabla r)\right)\notag\\
R(U,X,V,\nabla r) & = & \bar{k}\frac{w}{w^{\prime}}g\left(\nabla_{U}V,X\right)\label{eqnRUXU}\\
\left(\nabla_{\nabla r}R\right)(U,X,V,\nabla r) & = & \left(\frac{w^{\prime}}{w}+\bar{k}\frac{w}{w^{\prime}}\right)R(U,X,V,\nabla r)\label{eqnevolutionRUXU}\\
&  & +\bar{k}\frac{w}{w'}R(U,V,X,\nabla r).\notag
\end{eqnarray}
\end{proposition}

\begin{proof}
First note that 
\begin{equation*}
\mathrm{Hess}w=-\bar{k}w\mathrm{d}r\otimes\mathrm{d}r+w^{\prime}\mathrm{Hess}r.
\end{equation*}
Therefore the second fundamental form for the hypersurface $w^{-1}\left(r\right)$ is 
\begin{equation*}
\mathrm{II}=\mathrm{Hess}r=(w^{\prime})^{-1}\mathrm{Hess}w=\frac{w}{mw^{\prime}}\left(\mathrm{Ric}-\lambda g\right),
\end{equation*}
i.e., 
\begin{eqnarray*}
\mathrm{II}(X,X) & = & 0\\
\mathrm{II}(X,U) & = & 0\\
\mathrm{II}(U,U) & = & -\bar{k}\frac{w}{w^{\prime}}|U|^{2}.
\end{eqnarray*}
This implies 
\begin{eqnarray*}
R(X,\nabla r)\nabla r & = & 0\\
R(U,\nabla r)\nabla r & = & \bar{k}U,
\end{eqnarray*}
and
\begin{eqnarray*}
R(X,U,Y,\nabla r) & = & \left(\nabla_{U}\mathrm{II}\right)(X,Y)-\left(\nabla_{X}\mathrm{II}\right)(U,Y)\\
& = & \mathrm{II}(U,\nabla_{X}Y)\\
& = & -\bar{k}\frac{w}{w^{\prime}}g(\nabla_{X}Y,U).
\end{eqnarray*}
From the second Bianchi identity 
\begin{equation*}
\left(\nabla_{\nabla r}R\right)(X,U,Y,\nabla r)+\left(\nabla_{X}R\right)(U,\nabla r,Y,\nabla r)+\left(\nabla_{U}R\right)(\nabla r,X,Y,\nabla r)=0,
\end{equation*}
it follows then that
\begin{eqnarray*}
\left(\nabla_{\nabla r}R\right)(X,U,Y,\nabla r) & = & \left(\nabla_{X}R\right)(\nabla r,U,Y,\nabla r)-\left(\nabla_{U}R\right)(\nabla r,X,Y,\nabla r)\\
& = & -R\left(\nabla r,U,\nabla_{X}Y,\nabla r\right)+R\left(\nabla_{U}\nabla r,X,Y,\nabla r\right)+R\left(\nabla r,X,Y,\nabla_{U}\nabla r\right)\\
& = & \frac{1}{m}\left(\mathrm{Ric}(U,\nabla_{X}Y)-\lambda g(U,\nabla_{X}Y)\right)-\bar{k}\frac{w}{w^{\prime}}R(U,X,Y,\nabla r)-\bar{k}\frac{w}{w^{\prime}}R(\nabla r,X,Y,U)\\
& = & -\bar{k}g(U,\nabla_{X}Y)+\bar{k}\frac{w}{w^{\prime}}\left(R(X,U,Y,\nabla r)+R\left(Y,U,X,\nabla r\right)\right)\\
& = & \frac{w^{\prime}}{w}R(X,U,Y,\nabla r)+\bar{k}\frac{w}{w^{\prime}}\left(R(X,U,Y,\nabla r)+R(Y,U,X,\nabla r)\right).
\end{eqnarray*}
An analogous argument gives us the other equations:
\begin{eqnarray*}
R\left(U,X,V,\nabla r\right) & = & \left(\nabla_{X}\mathrm{II}\right)\left(U,V\right)-\left(\nabla_{U}\mathrm{II}\right)\left(X,V\right)\\
& = & \mathrm{II}\left(\nabla_{U}X,V\right)\\
& = & -\bar{k}\frac{w}{w^{\prime}}g\left(\nabla_{U}X,V\right)\\
& = & \bar{k}\frac{w}{w^{\prime}}g\left(X,\nabla_{U}V\right)
\end{eqnarray*}
and 
\begin{eqnarray*}
\left(\nabla_{\nabla r}R\right)(U,X,V,\nabla r) & = & \left(\nabla_{U}R\right)(\nabla r,X,V,\nabla r)-\left(\nabla_{X}R\right)(\nabla r,U,V,\nabla r)\\
& = & -R(\nabla r,\nabla_{U}X,V,\nabla r)-R(\nabla_{U}\nabla r,X,V,\nabla r)-R(\nabla r,X,V,\nabla_{U}\nabla r)\\
& = & \frac{1}{m}\left(\mathrm{Ric}(V,\nabla_{U}X)-\lambda g(V,\nabla_{U}X)\right)+\bar{k}\frac{w}{w^{\prime}}R(U,X,V,\nabla r)+\bar{k}\frac{w}{w'}R(\nabla r,X,V,U)\\
& = & -\bar{k}g\left(V,\nabla_{U}X\right)+\bar{k}\frac{w}{w^{\prime}}R(U,X,V,\nabla r)+\bar{k}\frac{w}{w'}R(U,V,X,\nabla r)\\
& = & \bar{k}g\left(\nabla_{U}V,X\right)+\bar{k}\frac{w}{w^{\prime}}R(U,X,V,\nabla r)+\bar{k}\frac{w}{w'}R(U,V,X,\nabla r)\\
& = & \left(\frac{w'}{w}+\bar{k}\frac{w}{w^{\prime}}\right)R(U,X,V,\nabla r)+\bar{k}\frac{w}{w'}R(U,V,X,\nabla r)\\
& = & \frac{w'}{w}R(U,X,V,\nabla r)+\bar{k}\frac{w}{w'}\left(R(U,X,V,\nabla r)-R(V,X,U,\nabla r)\right).
\end{eqnarray*}
These give us the desired equations.
\end{proof}

We can now finish the proof when the curvature is harmonic.

\begin{proof}[Proof of Corollary \protect\ref{corharmonic}] 
As the curvature is harmonic we have seen above that $R(X,U,X,\nabla r)=0$ and $R(U,X,U,\nabla r)=0$. When combined with the previous proposition, this tells us that the eigen-distributions are totally geodesic. This gives a splitting of the universal cover along the eigen-distributions of $P$. The results from section 2 tell us the metric is rigid.
\end{proof}

\smallskip{}

\begin{lemma} \label{lemparallelmupositive} 
If $M$ has non-empty boundary then the two eigen-distributions of $P$ are parallel. 
\end{lemma}

\begin{proof}
As $\nabla_{\nabla r}P=0$, we can assume that the eigenvector fields $X,Y,\ldots\in\mathcal{N}$ and $U,V,\ldots\in\mathcal{P}$ are parallel along $\nabla r$. The fact that the boundary is nonempty is used as an initial value for the curvatures. We assume that the boundary corresponds to the level set $r=0$. Specifically we see that Proposition \ref{propRadialCurvatureOperator} implies that $R(E_{1},E_{2},E_{3},\nabla r)=0$ on the boundary. If we set $X=Y$ in (\ref{eqnRXUX}) and (\ref{eqnevolutionRXUX}), then we obtain
\begin{eqnarray*}
\left(\nabla_{\nabla r}R\right)(X,U,X,\nabla r) & = & \left(\frac{w^{\prime}}{w}+2\bar{k}\frac{w}{w^{\prime}}\right)R(X,U,X,\nabla r)\\
R(X,U,X,\nabla r) & = & 0\quad\mathrm{at\:}r=0.
\end{eqnarray*}
So it follows that $R(X,U,X,\nabla r)=0$ on $M$. This shows that
$R\left(X,U,Y,\nabla r\right)$ is skew-symmetric in $X,Y$ and hence
that
\begin{equation*}
\left(\nabla_{\nabla r}R\right)(X,U,Y,\nabla r)=\frac{w^{\prime}}{w}R(X,U,Y,\nabla r).
\end{equation*}
We can then similarly conclude that $R(X,U,Y,\nabla r)$ vanishes as long as it vanishes on the boundary. Using (\ref{eqnRXUX}) again this shows that $\mathcal{N}$ is integrable as well as totally geodesic.

A similar argument works for $\mathcal{P}$. Setting $U=V$ in (\ref{eqnevolutionRUXU}) we have 
\begin{equation*}
\left(\nabla_{\nabla r}R\right)(U,X,U,\nabla r)=\left(\frac{w'}{w}+\bar{k}\frac{w}{w^{\prime}}\right)R(U,X,U,\nabla r).
\end{equation*}
So $R(U,X,U,\nabla r)$ vanishes as it vanishes on the boundary. This implies that $R(U,X,V,\nabla r)$ is skew-symmetric in $U$ and $V$, which in turn shows that 
\begin{eqnarray*}
R\left(U,V,X,\nabla r\right) & = & R(U,X,V,\nabla r)-R(V,X,U,\nabla r)\\
& = & 2R(U,X,V,\nabla r)
\end{eqnarray*}
and consequently 
\begin{equation*}
\left(\nabla_{\nabla r}R\right)(U,X,V,\nabla r)=\left(\frac{w'}{w}+3\bar{k}\frac{w}{w'}\right)R(U,X,V,\nabla r).
\end{equation*}
As $R(U,X,V,\nabla r)$ vanishes on the boundary, it vanishes everywhere showing that $\mathcal{P}$ is totally geodesic.
\end{proof}

\smallskip{}

\begin{lemma} \label{lemparallelmuzero} 
If $M$ has no boundary, i.e., $\bar{\mu}=0$, and we further assume that $\left|R\right|=o\left(\exp\left(dist(x,p)\sqrt{-\bar{k}}\right)\right)$, then the two eigenvalue distributions for $P$ are parallel.
\end{lemma}

\begin{proof}
We now turn to Case (\ref{casewexp}) where $w=\exp\left(\sqrt{-\bar{k}}r\right)$.
This means that equation (\ref{eqnevolutionRXUX}) reduces to 
\begin{eqnarray*}
D_{\nabla r}R(X,U,Y,\nabla r) & = & \sqrt{-\bar{k}}R(X,U,Y,\nabla r)-\sqrt{-\bar{k}}\left(R(X,U,Y,\nabla r)+R(Y,U,X,\nabla r)\right),
\end{eqnarray*}
i.e., 
\begin{equation*}
R(X,U,X,\nabla r)=R(X,U,X,\nabla r)|_{r=0}\exp\left(-\sqrt{-\bar{k}}r\right).
\end{equation*}
In particular we see again that $R(X,U,X,\nabla r)$ must vanish if we assume that it cannot grow as fast as $\exp\left(\sqrt{-\bar{k}}r\right).$ This again shows that $R\left(X,U,Y,\nabla r\right)$ is skew-symmetric in $X,Y$ and thus
\begin{eqnarray*}
D_{\nabla r}R(X,U,Y,\nabla r) & = & \sqrt{-\bar{k}}R(X,U,Y,\nabla r)-\sqrt{-\bar{k}}\left(R(X,U,Y,\nabla r)+R(Y,U,X,\nabla r)\right)\\
& = & \sqrt{-\bar{k}}R(X,U,Y,\nabla r).
\end{eqnarray*}
Using the growth assumption for $R(X,U,Y,\nabla r)$ this in turn shows that $R(X,U,Y,\nabla r)$ vanishes and then that $\mathcal{N}$ is totally geodesic.

Next we note that $R(U,X,V,\nabla r)=\bar{k}\frac{w}{w^{\prime}}g\left(\nabla_{U}V,X\right)$.
The right hand side can be calculated using Koszul's formula together with the metric decomposition \begin{eqnarray*}
g & = & \mathrm{d}r^{2}+g_{r}\\
& = & \mathrm{d}r^{2}+h_{0}+\frac{\left(w'\left(r\right)\right)^{2}}{\left(w'\left(0\right)\right)^{2}}h_{1},
\end{eqnarray*}
where $h_{0}$ and $h_{1}$ are the restrictions of $g_{0}$ to $\mathcal{N}$ and $\mathcal{P}$ at $r=0$ respectively.

We wish to calculate $g\left(\nabla_{U}V,X\right)$ and relate it to what happens at $r=0$. This requires that the fields we use commute with $\nabla r$. As our fields are chosen to be parallel along $\nabla r$ we simply have to switch to $X,w'U,w'V$ instead. After eliminating $\left(w'\right)^{2}$ on both sides this yields the formula 
\begin{eqnarray*}
g(\nabla_{U}V,X) & = & g_{r}(\nabla_{U}V,X)\\
& = & g_{0}(\nabla_{U}V,X)+\left(\frac{1}{2}-\frac{1}{2\left(w'\left(r\right)\right)^{2}}\right)g_{0}([V,U],X).
\end{eqnarray*}
Thus 
\begin{eqnarray*}
R(U,X,V,\nabla r) & = & \bar{k}\frac{w}{w'}\left(g_{0}(\nabla_{U}V,X)+\left(\frac{1}{2}- \frac{1}{2\left(w'\left(r\right)\right)^{2}}\right)g_{0}([V,U],X)\right)
\end{eqnarray*}
is forced to grow like $\left(w'\right)^{-2}=\left|\bar{k}\right|\exp\left(-2r\sqrt{-\bar{k}}\right)$ unless $\left[U,V\right]$ is perpendicular to $X$. As we have assumed that the curvature grows slower than $\exp\left(\sqrt{-\bar{k}}r\right)$ this shows that $\mathcal{P}$ is an integrable distribution.

Having shown that $\mathcal{N}$ is totally geodesic and its orthogonal distribution $\left\{ \nabla r\right\} \oplus\mathcal{P}$ is integrable it follows that the foliation with vertical space given by $\left\{ \nabla r\right\} \oplus\mathcal{P}$ is Riemannian. This means that we can use \cite{BlumenthalHebda} to conclude that there is a map $F:B\times H\rightarrow M$ such that
$F\left(B\times\left\{ q\right\} \right)$ is an integral manifold for $\mathcal{N}$ for all $q\in H$ and similarly $F\left(\left\{ p\right\} \times H\right)$ is an integral manifold for $\left\{ \nabla r\right\} \oplus\mathcal{P}$ for all $p\in B$. Below we shall show that the fibers are all isometric to each other and are in fact a simply connected hyperbolic space of constant curvature $\bar{k}.$

Note that if $X$ is chosen to be basic along this Riemannian foliation then $[X,U]\in\mathcal{P}$ for any $U\in\mathcal{P}$. As $\mathcal{N}$ is totally geodesic we have $\nabla_{X}U\in\mathcal{P}$, consequently also $\nabla_{U}X\in\mathcal{P}$. Let $U,V\in\mathcal{P}$ with $g(U,V)=0$, then
\begin{eqnarray*}
D_{\nabla r}R(V,U,U,V) & = & \left(\nabla_{V}R\right)(V,U,U,\nabla r)-\left(\nabla_{U}R\right)(V,U,V,\nabla
r)\\
& = & -R\left(\nabla_{V}V,U,U,\nabla r\right)-R\left(V,\nabla_{V}U,U,\nabla r\right)\\
&  & -R\left(V,U,\nabla_{V}U,\nabla r\right)-R\left(V,U,U,\nabla_{V}\nabla r\right)\\
&  & +R\left(\nabla_{U}V,U,V,\nabla r\right)+R\left(V,\nabla_{U}U,V,\nabla r\right)\\
&  & +R\left(V,U,\nabla_{U}V,\nabla r\right)+R\left(V,U,V,\nabla_{U}\nabla r\right)\\
& = & -\sqrt{-\bar{k}}g\left(\mathcal{N}\left(\nabla_{V}V\right),\nabla_{U}U\right)+R(\nabla r,U,U,\nabla
r)\mathrm{II}(V,V) \\
& & +\sqrt{-\bar{k}}g\left(\mathcal{N}\left(\nabla_{V}U\right),\nabla_{V}U\right) +\sqrt{-\bar{k}}g\left(\mathcal{N}\left(\nabla_{U}V\right),\nabla_{U}V\right) \\
& & - \sqrt{-\bar{k}}g\left(\mathcal{N}\left(\nabla_{U}U\right),\nabla_{V}V\right)-R(V,\nabla
r,V,\nabla r)\mathrm{II}(U,U)\\
&  & -2\sqrt{-\bar{k}}R\left(V,U,U,V\right)\\
& = &
-2\sqrt{-\bar{k}}g\left(\mathcal{N}\left(\nabla_{V}V\right),\mathcal{N}\left(\nabla_{U}U\right)\right)+ 2\sqrt{-\bar{k}}\left\vert
\mathcal{N}\left(\nabla_{V}U\right)\right\vert ^{2}\\
&  & -2\sqrt{-\bar{k}}R\left(V,U,U,V\right)+2\bar{k}\sqrt{-\bar{k}}g(U,U)g(V,V).
\end{eqnarray*}
Here the term from the second fundamental form
\begin{equation}
g\left(\mathcal{N}\left(\nabla_{V}V\right),\mathcal{N}\left(\nabla_{U}U\right)\right)-\left\vert \mathcal{N}\left(\nabla_{V}U\right)\right\vert ^{2}-\bar{k}g(U,U)g(V,V)\label{eqnKUVsecondff}
\end{equation}
is constant in the $\nabla r$ direction. Thus $R\left(V,U,U,V\right)$ will grow exponentially unless it is constant and precisely cancels the term in (\ref{eqnKUVsecondff}), i.e.,
\begin{equation*}
R\left(V,U,U,V\right)=\left\vert \mathcal{N}\left(\nabla_{V}U\right)\right\vert
^{2}-g\left(\mathcal{N}\left(\nabla_{V}V\right),\mathcal{N}\left(\nabla_{U}U\right)\right)+\bar{k}g(U,U)g(V,V).
\end{equation*}
This indicates that $\mathcal{P}$ is flat in $M$ and more importantly that the fibers are isometric to the simply-connected hyperbolic space with curvature $\bar{k}$ which we denote by $(H,h)$.

We then conclude that the structure $F:B\times H\rightarrow M$ gives us a metric decomposition of the form 
\begin{equation*}
F^{\ast}\left(g\right)=h_{0}+f_{b}^{\ast}\left(h\right)
\end{equation*}
where $f_{b}:H\rightarrow H$ is a family of isometries parametrized by the base space $B.$ From this it follows that any isometry of $\left(H,h\right)$ extends to an isometry of $M$ that fixes the horizontal space $B$ and maps a fiber to itself. Recall that we have a distance function $r$ such that $\nabla r$ is tangent to the fibers. This means that a geodesic $\gamma$ with velocity vector $\nabla r$ stays in the fiber. Using an isometry $f\in\mathrm{Iso}\left(H,h\right)$
then yields a new geodesic $f\circ\gamma$ with velocity $Df\left(\nabla r\right)$ that also stays in the fiber. As the isometry group of hyperbolic space is isotropic, i.e., the isotropy action is transitive on the unit sphere in the tangent space, it follows that any geodesic tangent
to a fiber will stay in that fiber, i.e., the fiber is totally geodesic.
\end{proof}

Now we finish the proof of Theorem \ref{thmradialQ}.

\begin{proof}[Proof of Theorem \protect\ref{thmradialQ}]
Suppose $(M,g,w)$ is a non-trivial $(\lambda,n+m)$-Einstein manifold. By the definition
of $Q$, the radial section curvature condition (\ref{eqnradialseccurvature}) is equivalent to $Q(\nabla w,\cdot,\cdot,\nabla w)=0$. From Lemma \ref{lemradialQtwoeigenvalues}, the eigenvalues of $P$ are either $0$ or $\lambda-\rho$. From Lemma \ref{lemparallelmupositive} and \ref{lemparallelmuzero}, the distribution for each eigenvalue is parallel. Hence the manifold splits as $(M,g)=(M_{1}\times M_{2},g_{1}+g_{2})$ such that $\mathrm{Ric}_{g_{1}}=\lambda g_{1}$ and $\mathrm{Ric}_{g_{2}}=\rho g_{2}$. So the rigidity of $M$ follows from Proposition \ref{proprigidnontrivial}.
\end{proof}

\begin{remark}
We have only used the assumption about exponential growth of curvature in the special case when $\bar{\mu}=0$ and the potential function is $w=\exp\left(\sqrt{-\bar{k}}r\right)$. Moreover,
from the proof of Lemma \ref{lemparallelmuzero}, we can see that in this case we get a stronger conclusion than rigidity. Namely that the manifold $M$ splits isometrically as $L\times H$ where $L$ is $\lambda$-Einstein and $H$ is the simply-connected hyperbolic space with Ricci curvature
$\rho$.
\end{remark}

\medskip{}

\section{Non-rigid $(\lambda,n+m)$-Einstein metrics on solvable Lie groups}

In this section, we consider $(\lambda,n+m)$-Einstein metrics on homogeneous spaces. The general case of the $(\lambda,n+m)$-Einstein equation is studied in \cite{HPWVirtual} where the function $w$ is also allowed to take negative values. Using this generalized equation, we obtain

\begin{theorem} \label{thmhomogeneousmu0}
If $(M,g)$ be a simply connected homogeneous non-trivial $(\lambda,m+n)$-Einstein manifold
with $m>1$ and $\lambda<0$, then one of the following cases holds.
\begin{enumerate}
\item $\mu<0$ and $M=S\times L$ as a Riemannian product where $S$ is a space form that is $\rho$-Einstein and $L$ is $\lambda$-Einstein. 
\item $\mu=0$ and there are no other solutions to the $\lambda$-Einstein equations with $\mu<0,$ then the isometry group $\mathrm{Iso}(M)$ contains a codimension one normal subgroup $H$ that acts transitively on the connected components of the level set of $w$. Moreover if $F$ is an isometry of $M$, then $w\circ F=Cw$ for some constant $C$.
\end{enumerate}
\end{theorem}

\begin{remark}
In Case (1) above when $\mu<0$, one can show that
$M$ is rigid by Proposition \ref{propwfunctions}, Proposition \ref{propNequaldim}
and Proposition \ref{propcriticalpoints}. In fact, $\mu<0$ implies
that $w$ has critical points. At a critical point $\nabla\mathrm{Ric}=0$
and the eigenvalue of $\mathrm{Ric}$ is either $\lambda$ or $\rho$.
Since the metric is homogeneous, this holds everywhere on $M$. So
the metric splits along the eigen-distributions of $\mathrm{Ric}$,
i.e., it is rigid.
\end{remark}

Theorem \ref{thmhomogeneousmu0} suggests that we are most likely
to find non-rigid examples on solvable Lie groups. Using left invariant
vector fields, we can rewrite the $(\lambda,n+m)$-Einstein equation.

\begin{lemma} Let $(G,g)$ be a simply-connected Lie group with left
invariant metric. Suppose it admits a non-trivial $(\lambda,n+m)$-Einstein
structure with $m>1$ and $\mu=0$. Then there exists a codimension
one normal subgroup $H\subset G$ and left invariant vector field
$X_{0}=\nabla r$, where $r$ is the signed distance function from
$H$. Furthermore the $(\lambda,n+m)$-Einstein equation can be written
as
\begin{equation}
\mathrm{Ric}(X,Y)-\frac{1}{2}\sqrt{m(\rho-\lambda)}\left(g([X_{0},X],Y)+ g([X_{0},Y],X)\right)-(\rho-\lambda)g(X_{0},X)g(X_{0},Y)=\lambda
g(X,Y),\label{eqnlambdan+mleftinv}
\end{equation}
for any two left invariant vector fields $X$ and $Y$.
\end{lemma}

\begin{proof}
From Proposition \ref{propwfunctions} and Theorem \ref{thmhomogeneousmu0}
we may assume that $w=e^{\sqrt{-\bar{k}}r}$ and the level hypersurface
of $w=1$ is a codimension one normal subgroup $H$ in $G$. Thus
we obtain a Riemannian submersion $G\rightarrow G/H$ which is also
a Lie algebra homomorphism. This shows that left invariant vector
fields on $G/H$ lift to left invariant vector fields on $G$ that
are perpendicular to $H.$ As $G/H=\mathbb{R}$ it follows that $\nabla r$
is a left invariant vector field on $G.$

Using Koszul's formula we note that
\begin{equation*}
\mathrm{Hess}r\left(X,Y\right)=g\left(\nabla_{X}X_{0},Y\right)= \frac{1}{2}\left(g([X_{0},X],Y)+g([X_{0},Y],X)\right)
\end{equation*}
when $X,Y$ are left invariant. As
\begin{equation*}
\mathrm{Hess}w=\sqrt{-\bar{k}}w\mathrm{Hess}r-\bar{k}w\mathrm{d}r^{2}
\end{equation*}
we obtain the desired form for the $(\lambda,n+m)$-Einstein equation
(\ref{eqnlambdan+m}).
\end{proof}

\begin{proposition}
Let $G$ be a unimodular Lie group with left
invariant metric $g$. If $(G,g)$ is a nontrivial $(\lambda,n+m)$-Einstein
manifold with $m>1$, $\lambda<0$ and $\mu=0$, then $G=H\times\mathbb{R}$
where $H$ is $\lambda$-Einstein.
\end{proposition}

\begin{proof}
Choose an orthonormal basis $\left\{ X_{i}\right\} _{i=0}^{n-1}$
of left invariant vector fields. Then from the equation (\ref{eqnlambdan+mleftinv})
we have
\begin{equation*}
\mathrm{Ric}(X_{i},X_{i})=\sqrt{m(\rho-\lambda)}g([X_{0},X_{i}],X_{i})+(\rho-\lambda)\delta_{0i}+\lambda
\end{equation*}
which implies that
\begin{equation*}
\mathrm{scal}=\sqrt{m(\rho-\lambda)}\mathrm{tr}(\mathrm{ad}_{X_{0}})+(n-1)\lambda+\rho
\end{equation*}
The definition of being unimodular is that $\mathrm{tr}(\mathrm{ad}_{Z})=0$
for all $Z,$ so it follows that
\begin{equation*}
(n-1)\lambda-(m-1)\rho=(n-1)\lambda+\rho.
\end{equation*}
In particular $\rho=0$ and $\mathrm{scal}=(n-1)\lambda$. The manifold
splits by Theorem \ref{thmscaln-1}.
\end{proof}

In the following we give an explicit construction of non-rigid $(\lambda,4+m)$-Einstein
metrics on solvable Lie groups with $\lambda<0$ for all $m>0$.

Let $\mathfrak{g}$ be a 4-dimensional non-unimodular Lie algebra
of the type $\mathfrak{g}=\mathfrak{a}+\left[\mathfrak{g},\mathfrak{g}\right],$
where $\mathfrak{a}$ is a 2-dimensional abelian subalgebra and $\left[\mathfrak{g},\mathfrak{g}\right]$
the derived subalgebra, which we also assume to be abelian. Let $g$
denote a left invariant metric on the corresponding simply connected
Lie group $G$ with $X_{0},X_{1},X_{2},X_{3}$ an orthonormal frame
of left invariant vector fields such that $\mathfrak{a}=\mathrm{span}\left\{ X_{0},X_{1}\right\} $
and $\left[\mathfrak{g},\mathfrak{g}\right]=\mathrm{span}\left\{ X_{2},X_{3}\right\} .$
The distribution $\mathfrak{a}$ defines a totally geodesic foliation
whose leaves are $\mathbb{R}^{2}.$ The distribution $\left[\mathfrak{g},\mathfrak{g}\right]$
is Riemannian with vanishing $A$ tensor and the intrinsic geometry
of the leaves are also $\mathbb{R}^{2}.$ In other words we have a
Riemannian submersion
\begin{equation*}
\begin{array}{ccc}
\mathbb{R}^{2} & \longrightarrow & G\\
&  & \downarrow\\
&  & \mathbb{R}^{2}
\end{array}
\end{equation*}
where the base and intrinsic geometry of the fibers are flat and
the $A$ tensor for the submersion vanishes. We shall make further
simplifying assumptions about the Lie algebra, but in general we see
that the $T$ tensor controls everything. The Ricci curvatures are
computed using the follows properties.

\begin{proposition}
The Ricci tensor preserves the distributions
$\mathfrak{g}=\mathfrak{a}+\left[\mathfrak{g},\mathfrak{g}\right]$.
And for $i,j=0,1$
\begin{equation*}
\mathrm{Ric}\left(X_{i},X_{j}\right)= \sum_{k=2}^{3}2g\left(\left[X_{i},X_{k}\right],\nabla_{X_{k}}X_{j}\right)+ g\left(\nabla_{X_{k}}X_{i},\nabla_{X_{k}}X_{j}\right)
\end{equation*}
and $k,l=2,3$
\begin{equation*}
\mathrm{Ric}\left(X_{k},X_{l}\right)=\sum_{i\neq
k,l}-g\left(\nabla_{X_{i}}X_{i},\nabla_{X_{k}}X_{l}\right)+ g\left(\nabla_{X_{k}}X_{i},\left[X_{i},X_{l}\right]\right)+ g\left(\nabla_{X_{l}}X_{i},\left[X_{i},X_{k}\right]\right)+g\left(\nabla_{X_{k}}X_{i},\nabla_{X_{l}}X_{i}\right)
\end{equation*}
\end{proposition}

\begin{proof}
This relies on our knowledge of the covariant derivatives, specifically
that the foliation defined by $\mathfrak{a}$ is totally geodesic.
The only possible nonvanishing terms up to permutations of indices
are terms of the form
\begin{equation*}
g\left(\nabla_{X_{i}}X_{j},X_{k}\right),\text{ where }k=0,1\text{ and }i,j=2,3.
\end{equation*}
These can be computed by observing
\begin{eqnarray*}
g\left(\nabla_{X_{i}}X_{j},X_{k}\right) & = &
\frac{1}{2}\left(g\left(\left[X_{k},X_{i}\right],X_{j}\right)+ g\left(\left[X_{k},X_{j}\right],X_{i}\right)\right)\\
& = & -g\left(X_{j},\nabla_{X_{i}}X_{k}\right).
\end{eqnarray*}
Note also that these quantities are all constant since the metric
is left invariant.

With this in mind we can do the calculations. For $i=0,1$
\begin{eqnarray*}
\mathrm{Ric}\left(X_{i}\right) & = & \sum_{k=1}^{4}R\left(X_{i},X_{k}\right)X_{k}\\
& = & \sum_{k=3}^{4}R\left(X_{i},X_{k}\right)X_{k}
\end{eqnarray*}
and
\begin{eqnarray*}
R\left(X_{i},X_{k}\right)X_{k} & = &
\nabla_{X_{i}}\nabla_{X_{k}}X_{k}-\nabla_{X_{k}}\nabla_{X_{i}}X_{k}-\nabla_{\left[X_{i},X_{k}\right]}X_{k}\\
& = &
-\nabla_{X_{k}}\left[X_{i},X_{k}\right]-\nabla_{X_{k}}\nabla_{X_{k}}X_{i}- \nabla_{\left[X_{i},X_{k}\right]}X_{k}\\
& = & -2\nabla_{X_{k}}\left[X_{i},X_{k}\right]-\nabla_{X_{k}}\nabla_{X_{k}}X_{i}\in\mathrm{span}\left\{
X_{0},X_{1}\right\} .
\end{eqnarray*}
It shows that the Ricci tensor preserves the distributions.

This also helps us calculate the Ricci curvatures when $i,j=0,1$
\begin{eqnarray*}
\mathrm{Ric}\left(X_{i},X_{j}\right) & = & \sum_{k=2}^{3}R\left(X_{i},X_{k},X_{k},X_{j}\right) \\
& = &
\sum_{k=2}^{3}g\left(-2\nabla_{X_{k}}\left[X_{i},X_{k}\right]-\nabla_{X_{k}}\nabla_{X_{k}}X_{i},X_{j}\right)\\
& = &
\sum_{k=2}^{3}2g\left(\left[X_{i},X_{k}\right],\nabla_{X_{k}}X_{j}\right)+ g\left(\nabla_{X_{k}}X_{i},\nabla_{X_{k}}X_{j}\right).
\end{eqnarray*}

When $k,l=2,3$ we obtain similar formulas
\begin{eqnarray*}
\mathrm{Ric}\left(X_{k},X_{l}\right) & = & \sum_{i\neq k,l}R\left(X_{k},X_{i},X_{i},X_{l}\right)\\
& = & \sum_{i\neq k,l}g\left(\nabla_{X_{k}}\nabla_{X_{i}}X_{i},X_{l}\right)- g\left(\nabla_{X_{i}}\nabla_{X_{k}}X_{i},X_{l}\right)- g\left(\nabla_{\left[X_{k},X_{i}\right]}X_{i},X_{l}\right)\\
& = & \sum_{i\neq k,l}-g\left(\nabla_{X_{i}}X_{i},\nabla_{X_{k}}X_{l}\right)- g\left(\nabla_{X_{i}}\nabla_{X_{k}}X_{i},X_{l}\right)- g\left(\nabla_{\left[X_{k},X_{i}\right]}X_{i},X_{l}\right)\\
& = & \sum_{i\neq k,l}-g\left(\nabla_{X_{i}}X_{i},\nabla_{X_{k}}X_{l}\right)+ g\left(\nabla_{X_{k}}X_{i},\nabla_{X_{i}}X_{l}\right)- g\left(\nabla_{X_{l}}X_{i},\left[X_{k},X_{i}\right]\right)\\
& = & \sum_{i\neq k,l}-g\left(\nabla_{X_{i}}X_{i},\nabla_{X_{k}}X_{l}\right)+ g\left(\nabla_{X_{k}}X_{i},\left[X_{i},X_{l}\right]\right)+ g\left(\nabla_{X_{l}}X_{i},\left[X_{i},X_{k}\right]\right)\\
&  & +g\left(\nabla_{X_{k}}X_{i},\nabla_{X_{l}}X_{i}\right)
\end{eqnarray*}
that finishes the proof.
\end{proof}

Next we need to determine the structure constants. We proceed as in
the excellent reference \cite{Jensen} also using \cite{Milnor} to
find an appropriate 3 dimensional non-unimodular solvable Lie algebra
satisfying
\begin{eqnarray*}
[X_{1}, X_{2}] & = & \alpha X_{2}+\beta X_{3}\\
{}[X_{1},X_{3}] & = & \gamma X_{2}+(2-\alpha)X_{3},
\end{eqnarray*}
where $\alpha$, $\beta$ and $\gamma$ are constants. We will make
the further simplifying assumption that $\gamma=\beta\neq0.$ Thus
\begin{eqnarray*}
\nabla_{X_{2}}X_{1} & = & -\alpha X_{2}-\frac{1}{2}\beta X_{3}\\
\nabla_{X_{3}}X_{1} & = & -\beta X_{2}-\left(2-\alpha\right)X_{3}.
\end{eqnarray*}

There are similar constraints for the structure constants
\begin{equation*}
[X_{0},X_{i}]=F_{ij}X_{j},\quad\mbox{ for }i,j=2,3,
\end{equation*}
that yield the other relevant covariant derivatives:
\begin{eqnarray*}
\nabla_{X_{2}}X_{0} & = & -F_{22}X_{2}-\frac{1}{2}\left(F_{23}+F_{32}\right)X_{3}\\
\nabla_{X_{3}}X_{0} & = & -\frac{1}{2}\left(F_{23}+F_{32}\right)X_{2}-F_{33}X_{3}.
\end{eqnarray*}
However, we also need to see how they interact with the other structure
constants. Checking the Jacobi identity shows

\begin{proposition}
The constants $\alpha$, $\beta$ and $F_{ij}$ satisfy the following identities:
\begin{eqnarray*}
\beta F_{22}+2(1-\alpha)F_{23}-\beta F_{33} & = & 0\\
F_{23} & = & F_{32}.
\end{eqnarray*}
\end{proposition}

\begin{proof}
To verify the Jacobi identity we only have to consider three vectors
$X_{0}$, $X_{i}$ and $X_{j}$ for $1\leq i<j\leq3,$ i.e., the equations
\begin{eqnarray*}
\left[\left[X_{0},X_{1}\right],X_{2}\right]+\left[\left[X_{2},X_{0}\right],X_{1}\right]+ \left[\left[X_{1},X_{2}\right],X_{0}\right] & = & 0\\
\left[\left[X_{0},X_{1}\right],X_{3}\right]+\left[\left[X_{3},X_{0}\right],X_{1}\right]+ \left[\left[X_{1},X_{3}\right],X_{0}\right] & = & 0\\
\left[\left[X_{0},X_{2}\right],X_{3}\right]+\left[\left[X_{3},X_{0}\right],X_{2}\right]+ \left[\left[X_{2},X_{3}\right],X_{0}\right] & = & 0
\end{eqnarray*}
which can be reduced to
\begin{eqnarray*}
-\left[F_{2j}X_{j},X_{1}\right]+\left[\alpha X_{2}+\beta X_{3},X_{0}\right] & = & 0\\
-\left[F_{3j}X_{j},X_{1}\right]+\left[\beta X_{2}+(2-\alpha)X_{3},X_{0}\right] & = & 0
\end{eqnarray*}
or more explicitly
\begin{equation*}
F_{22}\left(\alpha X_{2}+\beta X_{3}\right)+F_{23}\left(\beta
X_{2}+(2-\alpha)X_{3}\right)-\alpha\left(F_{22}X_{2}+F_{23}X_{3}\right)- \beta\left(F_{32}X_{2}+F_{33}X_{3}\right)=0
\end{equation*}
\begin{equation*}
F_{32}\left(\alpha X_{2}+\beta X_{3}\right)+F_{33}\left(\beta
X_{2}+(2-\alpha)X_{3}\right)-\beta\left(F_{22}X_{2}+F_{23}X_{3}\right)- \left(2-\alpha\right)\left(F_{32}X_{2}+F_{33}X_{3}\right)=0.
\end{equation*}
This finally yields
\begin{eqnarray*}
\beta F_{23}-\beta F_{32} & = & 0\\
\beta F_{22}+(2-2\alpha)F_{23}-\beta F_{33} & = & 0\\
-\beta F_{22}-(2-2\alpha)F_{32}+\beta F_{33} & = & 0\\
-\beta F_{23}+\beta F_{32} & = & 0.
\end{eqnarray*}
This immediately gives the relationships.
\end{proof}

We now wish to solve the $\left(\lambda,4+m\right)$-Einstein equations
on the Lie group $G$. Let $H$ be the Lie subgroup corresponding
to the Lie subalgebra spanned by $\left\{ X_{1},X_{2},X_{3}\right\} $
and $r$ is the distance function from $H$. Since $\nabla_{X_{0}}X_{0}=0$
and $X_{0}$ is the unit normal vector field of $H$, we have $X_{0}=\nabla r$.
From the $(\lambda,n+m)$-Einstein equation (\ref{eqnlambdan+mleftinv})
in terms of left invariant vector fields, we have

\begin{lemma}
The $\left(\lambda,4+m\right)$-Einstein equations
on $G$ can be written as
\begin{eqnarray*}
\mathrm{Ric}(X_{0},X_{0}) & = & \rho\\
\mathrm{Ric}(X_{1},X_{1}) & = & \lambda\\
\mathrm{Ric}(X_{1},X_{0}) & = & 0\\
\mathrm{Ric}(X_{2},X_{2})-\sqrt{m\left(\rho-\lambda\right)}F_{22} & = & \lambda\\
\mathrm{Ric}(X_{3},X_{3})-\sqrt{m\left(\rho-\lambda\right)}F_{33} & = & \lambda\\
\mathrm{Ric}(X_{2},X_{3})-\sqrt{m\left(\rho-\lambda\right)}F_{32} & = & 0.
\end{eqnarray*}
\end{lemma}

These equations impose further constraints on the structure constants
and as we shall see the last three of them are equivalent when we
know that the first three hold.

\begin{theorem} \label{thmnonrigidexample}
Given $m>0$ and $\alpha,\beta\in\mathbb{R}$
with $\beta\ne0$ and $\left(\alpha-1\right)^{2}+\beta^{2}\neq1$,
we obtain a left invariant $\left(\lambda,4+m\right)$-Einstein metric
on $G$ with $\mu=0.$ Moreover, this structure is not rigid.
\end{theorem}

\begin{proof}
We start by observing that if we have such a solution then
\begin{eqnarray*}
\mathrm{Ric}(X_{1},X_{0}) & = & -\frac{2(-\alpha+\alpha^{2}+ \beta^{2})F_{32}+2\beta F_{33}}{\beta}=0,\\
\mathrm{Ric}(X_{0},X_{0}) & = &
-\frac{2\left(F_{32}\right)^{2}(2-4\alpha^{3}+\alpha^{4}+3\beta^{2}+ \beta^{4}-2\alpha(3+2\beta^{2})+\alpha^{2}(7+2\beta^{2}))}{\beta^{2}}=\rho,\\
\mathrm{Ric}(X_{1},X_{1}) & = & -2(2-2\alpha+\alpha^{2}+\beta^{2})=\lambda,\\
\mathrm{Ric}(X_{2},X_{2}) & = &
-\frac{2(\alpha\beta^{2}+\left(F_{32}\right)^{2}(2-5\alpha^{3}+\alpha^{4}+ 3\beta^{2}+\beta^{4}+\alpha^{2}(9+2\beta^{2})-\alpha(7+5\beta^{2})))}{\beta^{2}}\\
& = & \lambda-\frac{\left(F_{32}\right)(2-3\alpha+\alpha^{2}+\beta^{2})}{\beta}\sqrt{m(\rho-\lambda)},\\
\mathrm{Ric}(X_{2},X_{3}) & = &
\frac{2(-\beta^{2}+\left(F_{32}\right)^{2}(1-2\alpha+\alpha^{2}+\beta^{2}))}{\beta} =\left(F_{32}\right)\sqrt{m(\rho-\lambda)},\\
\mathrm{Ric}(X_{3},X_{3}) & = &
-\frac{2((2-\alpha)\beta^{2}+\left(F_{32}\right)^{2}(-3\alpha^{3}+ \alpha^{4}+\beta^{2}+\beta^{4}+\alpha^{2}(3+2\beta^{2})-\alpha(1+3\beta^{2})))}{\beta^{2}}\\
& = & \lambda-\frac{\left(F_{32}\right)(-\alpha+\alpha^{2}+\beta^{2})}{\beta}\sqrt{m(\rho-\lambda)}.
\end{eqnarray*}
The first equation forces the relationship
\begin{equation*}
F_{33}=-\frac{-\alpha+\alpha^{2}+\beta^{2}}{\beta}F_{32}.
\end{equation*}
If we introduce the notation
\begin{equation*}
z=\frac{F_{32}}{\beta}
\end{equation*}
then we have to solve the equations
\begin{eqnarray*}
\mathrm{Ric}(X_{0},X_{0}) & = &
-2z^{2}(2-4\alpha^{3}+\alpha^{4}+3\beta^{2}+\beta^{4}-2\alpha(3+2\beta^{2})+\alpha^{2}(7+2\beta^{2}))=\rho,\\
\mathrm{Ric}(X_{1},X_{1}) & = & -2(2-2\alpha+\alpha^{2}+\beta^{2})=\lambda,\\
\mathrm{Ric}(X_{2},X_{2}) & = &
-2\alpha-2z^{2}(2-5\alpha^{3}+\alpha^{4}+3\beta^{2}+\beta^{4}+\alpha^{2}(9+2\beta^{2})-\alpha(7+5\beta^{2}))\\
& = & \lambda-z(2-3\alpha+\alpha^{2}+\beta^{2})\sqrt{m(\rho-\lambda)},\\
\frac{1}{\beta}\mathrm{Ric}(X_{2},X_{3}) & = &
-2+2z^{2}(1-2\alpha+\alpha^{2}+\beta^{2})=z\sqrt{m(\rho-\lambda)},\\
\mathrm{Ric}(X_{3},X_{3}) & = &
-2(2-\alpha)\beta^{2}-2z^{2}(-3\alpha^{3}+\alpha^{4}+\beta^{2}+\beta^{4}+\alpha^{2}(3+2\beta^{2})- \alpha(1+3\beta^{2}))\\
& = & \lambda-z(-\alpha+\alpha^{2}+\beta^{2})\sqrt{m(\rho-\lambda)}.
\end{eqnarray*}
Note that the first equation combined with the second yields
\begin{eqnarray*}
\rho & = &
-2z^{2}(2-4\alpha^{3}+\alpha^{4}+3\beta^{2}+\beta^{4}-2\alpha(3+2\beta^{2})+\alpha^{2}(7+2\beta^{2}))\\
& = & -2z^{2}\left(2-2\alpha+\alpha^{2}+\beta^{2}\right)\left(1-2\alpha+\alpha^{2}+\beta^{2}\right)\\
& = & -z^{2}\frac{\lambda\left(\lambda+2\right)}{2}
\end{eqnarray*}
showing
\begin{eqnarray*}
\rho-\lambda & = & -z^{2}\frac{\lambda\left(\lambda+2\right)}{2}-\lambda\\
& = & -\lambda\left(\frac{\lambda+2}{2}z^{2}+1\right).
\end{eqnarray*}
The last three equations turn out to be identical. Specifically,
the fourth equation is the same as the third if we factor by $(2-3\alpha+\alpha^{2}+\beta^{2})$
and the same as the fifth if we factor by $(-\alpha+\alpha^{2}+\beta^{2})$.
Thus these three equations reduce to
\begin{equation*}
-2+2z^{2}(1-2\alpha+\alpha^{2}+\beta^{2})=z\sqrt{m(\rho-\lambda)}.
\end{equation*}
or
\begin{equation*}
-2-\left(\lambda+2\right)z^{2}=z\sqrt{-\lambda m\left(\frac{\lambda+2}{2}z^{2}+1\right)}.
\end{equation*}
This calculates $z$ as a function of $\lambda.$ Either
\begin{equation*}
z^{2}=\frac{-2}{\lambda+2}
\end{equation*}
forcing $\rho=\lambda,$ or
\begin{equation}
z^{2}=\frac{2}{-\frac{\lambda}{2}m-\left(\lambda+2\right)}=\frac{-4}{\lambda\left(m+2\right)+4}\label{eqnzsquare}
\end{equation}
which is equivalent to the first when $m=0$. Only the latter case
with $m>0$ gives us a nontrivial solution. If we substitute it into
the equation for $\rho$ we are left with the two equations
\begin{eqnarray}
-2(2-2\alpha+\alpha^{2}+\beta^{2}) & = & \lambda,\label{eqnlambdaalphabeta}\\
\frac{2\lambda\left(\lambda+2\right)}{\lambda\left(m+2\right)+4} & = & \rho.\notag
\end{eqnarray}
Note that we always have
\begin{equation*}
2-2\alpha+\alpha^{2}+\beta^{2}\geq1
\end{equation*}
so $\lambda<-2$.

Finally we do a few calculations to show that this is not rigid. Since
$\lambda\neq\rho$ it can only be rigid when the other two eigenvalues
are also $\lambda,\rho.$

We have
\begin{eqnarray}
z\sqrt{m(\rho-\lambda)} & = & -2-\left(\lambda+2\right)\frac{-4}{\lambda\left(m+2\right)+4}\notag\\
& = &
\frac{-2\left(\lambda\left(m+2\right)+4\right)+4\left(\lambda+2\right)}{\lambda\left(m+2\right)+4}= \frac{-2\lambda
m}{\lambda\left(m+2\right)+4}\notag\label{eqnzradical}\\
& = & -\frac{m\rho}{\lambda+2}.\notag
\end{eqnarray}
Thus
\begin{eqnarray*}
\mathrm{scal} & = & \lambda+\rho+\mathrm{Ric}(X_{2},X_{2})+\mathrm{Ric}(X_{3},X_{3})\\
& = & \lambda+\rho+2\lambda+\left(\lambda+2\right)z\sqrt{m(\rho-\lambda)}\\
& = & 3\lambda-\left(m-1\right)\rho.
\end{eqnarray*}
So the only possibility for having a rigid solution is when
\begin{equation*}
3\lambda-\left(m-1\right)\rho=2\rho+2\lambda
\end{equation*}
which implies
\begin{equation*}
\lambda=\left(m+1\right)\rho=\frac{2\lambda\left(\lambda+2\right)\left(m+1\right)}{\lambda\left(m+2\right)+4}
\end{equation*}
or
\begin{equation*}
\frac{2\left(\lambda+2\right)\left(m+1\right)}{\lambda\left(m+2\right)+4}=1,\quad\mbox{i.e.,}\quad\lambda=-4.
\end{equation*}
It follows that
\begin{equation*}
\rho=-\frac{4}{m+1}
\end{equation*}
and
\begin{equation*}
\mathrm{Ric}(X_{2},X_{2})=-4+\frac{2m}{m+1}\left(2-\alpha\right),\quad\mathrm{Ric}(X_{2},X_{3})= -\frac{2m}{m+1}\beta,\quad\mathrm{Ric}(X_{3},X_{3})=-4+\frac{2m}{m+1}\alpha.
\end{equation*}
The principal Ricci curvatures are $\lambda$ and $\rho$ if $\left(\alpha-1\right)^{2}+\beta^{2}=1$.
In this case, the metric splits along the eigendistributions of $\mathrm{Ric}$.
\end{proof}

Finally we consider the convergence of the metrics in Theorem \ref{thmnonrigidexample}
as $m\rightarrow\infty$. We will see that the Ricci soliton structure
naturally appears in the limiting metric. First we recall

\begin{definition}
A left invariant metric $g$ on a simply-connected
nilpotent(or solvable) Lie group is called a \emph{nilsoliton}(or
\emph{solvsoliton}) if the Ricci tensor satisfies the following equation
\begin{equation}
\mathrm{Ric}(g)=cI+D\label{eqnalgebrasoliton}
\end{equation}
where $c$ is a constant and $D\in\mathrm{Der}(\mathfrak{g})$ is
a derivation.
\end{definition}

These manifolds are closely related to Ricci flows and Einstein metrics
on noncompact homogeneous spaces, see \cite{LauretNil,LauretSol,LauretRFNil}
and references therein.

\begin{theorem} \label{thmnonrigidconvergence}
Let $\left(G,g_{m}\right)$
be the $(\lambda,4+m)$-Einstein metrics in Theorem \ref{thmnonrigidexample}.
As $m\rightarrow\infty$, they converge in $C^{\infty}$ to the Riemannian
product $\left(\mathbb{R}\times H,\mathrm{d}t^{2}+g_{0}\right)$ where
$(H,g_{0})$ is a three dimensional solvsoliton.
\end{theorem}

\begin{proof}
Recall that $\left\{ X_{0},X_{1},X_{2},X_{3}\right\} $ is an orthonormal
frame of left invariant vector fields on $G$ and then any left invariant
metric $g$ is determined by the structure constants $C_{ijk}=g([X_{i},X_{j}],X_{k})$
for $i,j,k=0,\ldots,3$. That the metrics $g_{m}$ converge to $g_{\infty}$
in $C^{\infty}$ is equivalent to that the structure constants $C_{ijk}^{m}$
of $g_{m}$ converge to those of $g_{\infty}$, see for example \cite[Proposition 2.8]{LauretRFNil}
when $G$ is a nilpotent group. Alternatively, the $C^{\infty}$ convergence
of the structure constants implies $C^{\infty}$ convergence of the
Levi-Civita connections and then the convergence of the metrics follows,
see for example \cite[Section 6]{GlickensteinPayne}.

From the equation (\ref{eqnlambdaalphabeta}) of $\lambda$ in terms
of $\alpha$ and $\beta$, since $\lambda$ is unchanged for the metrics
$g_{m}$, $\alpha$ and $\beta$ are also unchanged. From the equation
(\ref{eqnzsquare})
\begin{equation*}
\lim_{m\rightarrow\infty}z^{2}=\lim_{m\rightarrow\infty}\frac{-4}{\lambda\left(m+2\right)+4}=0
\end{equation*}
It follows that $F_{32}$ and $F_{33}$ converge to zero, i.e., $[X_{0},X_{i}]=0$
for all $i=0,\ldots,3$ as $m\rightarrow\infty$. So the Lie algebra
$\mathfrak{g}$ corresponding to $g_{\infty}$ on $G$ splits as direct
sum of $\mathbb{R}X_{0}\oplus\mathfrak{h}$ where $\mathfrak{h}$
is spanned by $X_{1},X_{2}$ and $X_{3}$, and the limiting metric
$g_{\infty}=\mathrm{d}t^{2}+g_{0}$ where $g_{0}$ is the restriction
on the Lie group $H$ whose Lie algebra is $\mathfrak{h}$. On $H$,
the nonzero Lie brackets are given by
\begin{equation*}
\mathrm{ad}_{X_{1}}
\begin{pmatrix}X_{2}\\
X_{3}
\end{pmatrix}=
\begin{pmatrix}\alpha & \beta\\
\beta & 2-\alpha
\end{pmatrix}
\begin{pmatrix}X_{2}\\
X_{3}
\end{pmatrix}.
\end{equation*}
A computation shows that the nonzero Ricci curvatures are
\begin{eqnarray*}
\mathrm{Ric}(X_{1},X_{1}) & = & \lambda\\
\mathrm{Ric}(X_{2},X_{2}) & = & -2\alpha\\
\mathrm{Ric}(X_{2},X_{3}) & = & -2\beta\\
\mathrm{Ric}(X_{3},X_{3}) & = & -2(2-\alpha).
\end{eqnarray*}
Let $D$ be the derivation on $\mathfrak{h}$ as $D(X_{1})=0$ and
$D(X_{i})=-\lambda X_{i}-2\mathrm{ad}_{X_{1}}(X_{i})$ for $i=2,3$,
then we have $\mathrm{Ric}=\lambda I+D$ which shows $(H,g_{0})$
is a solvsoliton. The solvsoliton structure can also be obtained by
viewing $\mathfrak{h}$ as an extension of the abelian Lie algebra
$\mathrm{span}_{\mathbb{R}}\left\{ X_{2},X_{3}\right\} $ by $\mathbb{R}X_{1}$,
see the construction in \cite[Proposition 4.3]{LauretSol}. Note the
solvsoliton metric is not Einstein since $\beta\neq0$.

This derivation can be constructed directly as a limit in the following
way: we know that
\begin{equation*}
\frac{m}{w}\nabla\nabla w=\sqrt{m\left(\rho-\lambda\right)}\nabla X_{0}+(\rho-\lambda)X_{0}^{\flat}\otimes
X_{0}^{\flat},
\end{equation*}
where $X_{0}^{\flat}$ is a $1$-form dual to $X_{0}$. And the only
nontrivial covariant derivatives for $X_{0}$ are
\begin{eqnarray*}
\nabla_{X_{2}}X_{0} & = & \left(2-3\alpha+\alpha^{2}+\beta^{2}\right)-\beta zX_{3}\\
& = & \left(-\frac{\lambda}{2}-\alpha\right)zX_{2}-\beta zX_{3}\\
\nabla_{X_{3}}X_{0} & = & -\beta zX_{2}+\left(-\alpha+\alpha^{2}+\beta^{2}\right)zX_{3}\\
& = & -\beta zX_{2}+\left(-\frac{\lambda}{2}+2-\alpha\right)zX_{3}.
\end{eqnarray*}
Next note that
\begin{equation*}
\lim_{m\rightarrow\infty}z\sqrt{m\left(\rho-\lambda\right)}= \lim_{m\rightarrow\infty}z\sqrt{\frac{-4m\left(\rho-\lambda\right)}{\lambda\left(m+2\right)+4}}=-2
\end{equation*}
so it follows that $\frac{m}{w}\nabla\nabla w$ converges to the derivation $D$ as $m\rightarrow\infty.$

A similar limit argument also shows that the Ricci curvatures of $g_{m}$ converge to those of the limiting metric $g_{\infty}$.
\end{proof}

\medskip{}

\appendix

\section{An integrability criterion of eigen-distribution of $\mathrm{Ric}$ }

In this appendix we will show the eigen-distributions of $P$(or $\mathrm{Ric}$) in Theorem \ref{thmradialQ} are integrable without using the curvature growth assumption. It follows from the following general integrability criterion of eigen-distribution of Ricci curvature.

\begin{theorem}\label{thmRiccidistributionintegrable} 
Suppose $g_{t}=h_{0}+th_{1}(t>0)$ is a family of Riemannian metric on $\Sigma$ such that \begin{eqnarray*}
T\Sigma & = & \mathcal{N}\oplus\mathcal{P},\\
g_{1}|_{\mathcal{N}} & = & h_{0},\\
g_{1}|_{\mathcal{P}} & = & h_{1},
\end{eqnarray*}
where $\mathcal{N}$ and $\mathcal{P}$ are two orthogonal distributions. If the Ricci curvature of $g_{t}$ satisfies the following condition
\begin{equation*}
\mathrm{Ric}^{t}=\lambda h_{0}+\sigma h_{1}
\end{equation*}
where $\lambda\ne\sigma$ are two constants, then the two distributions $\mathcal{N}$ and $\mathcal{P}$ are integrable. 
\end{theorem}

Before proving Theorem \ref{thmRiccidistributionintegrable}, we shows that in Theorem \ref{thmradialQ} the restriction of the metric $g$ on the level sets of $\nabla r$ satisfies the assumptions in Theorem \ref{thmRiccidistributionintegrable}. Then the integrability of the eigen-distributions of $P$ follows directly.

We denote orthonormal frames of $\mathcal{N}$ and $\mathcal{P}$ by $\left\{ E_{i}\right\} $ and $\left\{ E_{a}\right\}$ respectively and, use $i,j,k,\ldots$ for the indices of $\mathcal{N}$, $a,b,c,\ldots$ for $\mathcal{P}$ and $\alpha,\beta,\gamma,\ldots$. The vector fields in $\mathcal{N}$ and $\mathcal{P}$ are denoted by $X,Y,Z,\ldots$ and $U,V,W,\ldots$ respectively. If $A$ is a vector field in $TM$, then $\mathcal{N}A$ and $\mathcal{P}A$ denote its projections.

Recall from Lemma \ref{lemradialQtwoeigenvalues} the eigenvalue of $P$ is either $0$ or $\lambda-\rho$. Thus we have the orthogonal decomposition of the tangent space of the level sets $\Sigma$ of $r$ as 
\begin{eqnarray*}
M & = & I\times\Sigma,\\
g & = & \mathrm{d}r^{2}+g_{r},\\
T\Sigma & = & \mathcal{N}\oplus\mathcal{P}.
\end{eqnarray*}
Instead of parametrizing by $r$, it will turn out to be more convenient to parametrize the metrics on the level sets by $t=\left(w^{\prime}\right)^{2}(r)$. We consider the one parameter family of metrics $g_{t}$ on $\Sigma$. In the next proposition we compute the Ricci curvature of this family.

\begin{proposition}\label{propriccihypersurface}
The metrics $(\Sigma, g_{t})$ satisfy
\begin{eqnarray*}
g_{t} & = & h_{0}+th_{1}\\
\mathrm{Ric}^{t} & = & \lambda h_{0}+\sigma h_{1}\\
\mathrm{scal}^{t} & = & \lambda\dim\mathcal{N}+\frac{\sigma}{t}\dim\mathcal{P},
\end{eqnarray*}
where $\mathrm{Ric}^{t}$ and $\mathrm{scal}^{t}$ denote the curvatures of $g_{t}$, $h_{0},h_{1}$ are metrics on $\mathcal{N}$ and $\mathcal{P}$ respectively, and $\sigma=\left(\dim\mathcal{P}-1\right)\bar{k}\bar{\mu}$.
\end{proposition}

\begin{remark}
By multiplying $w$ by a constant, we can always assume that $\sigma\neq\lambda$.
\end{remark}

\begin{remark}
In Proposition \ref{propriccihypersurface}, instead of considering the metrics on a family of hypersurfaces in $M$ we are fixing just one hypersurface $\Sigma$ and then pulling the metrics on the other hypersurfaces back via the flow generated by $\nabla r.$ If $\left\{ E_{i}\right\} $, $\left\{ E_{a}\right\} $ are orthonormal frames of $\mathcal{N}$, $\mathcal{P}$ respectively with respect to $g_{t}$, then the formula tells us that $\left\{ e_{i}=E_{i}\right\} $ and $\left\{ e_{a}=\sqrt{t}E_{a}\right\} $ are orthonormal frames with respect to $g_{1}$.
\end{remark}

\begin{proof}
Write the metric splitting on $\Sigma$ corresponding to some regular value $w_{0}$ as
\begin{equation*}
g_{0}=g_{\mathcal{N}}+g_{\mathcal{P}}.
\end{equation*}
From the $\left(\lambda,n+m\right)$-Einstein equation we know that
\begin{eqnarray*}
\mathcal{L}_{\nabla w}g_{\mathcal{N}} & = & 0,\\
\mathcal{L}_{\nabla w}g_{\mathcal{P}} & = & -2w\bar{k}g_{\mathcal{P}},
\end{eqnarray*}
and using $w=w\left(r\right)$ where $w^{\prime\prime}=-\bar{k}w$ we obtain
\begin{eqnarray*}
\mathcal{L}_{\nabla r}g_{\mathcal{N}} & = & 0,\\
\mathcal{L}_{\nabla r}g_{\mathcal{P}} & = & -2\bar{k}\frac{w}{w^{\prime}}g_{\mathcal{P}},
\end{eqnarray*}
so the metric has the form
\begin{equation*}
g=\mathrm{d}r^{2}+h_{0}+\left(w^{\prime}\right)^{2}h_{1}
\end{equation*}
where $h_{0}$ and $h_{1}$ are restrictions of $g$ on $\mathcal{N}$
and $\mathcal{P}$ at the regular value $w_{0}$ respectively.

As in the proof of Proposition \ref{propRadialCurvatureOperator}, we have the following formulas for the second fundamental form:
\begin{eqnarray*}
\mathrm{II}(X,X) & = & 0\\
\mathrm{II}(X,U) & = & 0\\
\mathrm{II}(U,U) & = & -\bar{k}\frac{w}{w^{\prime}}|U|^{2}.
\end{eqnarray*}

The radial flatness assumption
\begin{equation*}
0=Q(\nabla w,E,E,\nabla w)=R(E,\nabla w,\nabla w,E)+\frac{|\nabla w|^{2}}{m}(\mathrm{Ric}(E,E)-\lambda|E|^{2}).
\end{equation*}
then tells us that the radial curvatures satisfy
\begin{eqnarray*}
R\left(X,\nabla r\right)\nabla r & = & 0,\\
R\left(U,\nabla r\right)\nabla r & = & \bar{k}U,
\end{eqnarray*}
and $\mathrm{Ric}(\nabla r,\nabla r)=\rho$ implies that
\begin{equation*}
\rho=\bar{k}\dim\mathcal{P}.
\end{equation*}
This implies that for any two unit vectors $X\in\mathcal{N}$ and
$U\in\mathcal{P}$, we have
\begin{eqnarray*}
\mathrm{Ric}_{g_{r}}(X,X) & = & \lambda g_{r}(X,X)=\lambda\\
\mathrm{Ric}_{g_{r}}(X,U) & = & 0\\
\mathrm{Ric}_{g_{r}}(U,U) & = &
\rho-\bar{k}+\left(\dim\mathcal{P}-1\right)\bar{k}^{2}\frac{w^{2}}{(w^{\prime})^{2}}= \left(\dim\mathcal{P}-1\right)\bar{k}\left(\bar{k}\frac{w^{2}}{(w^{\prime})^{2}}+1\right)\\
& = & \left(\dim\mathcal{P}-1\right)\bar{k}\bar{\mu}(w^{\prime})^{-2},
\end{eqnarray*}
where
\begin{eqnarray*}
\bar{\mu} & = & \bar{k}w^{2}+|\nabla w|^{2}\\
& = & \bar{k}w^{2}+\left(w^{\prime}\right)^{2},
\end{eqnarray*}
which gives us
\begin{eqnarray*}
\mathrm{Ric}^{t}(E,F) & = & \mathrm{Ric}_{g_{t}}(E,F)\\
& = & \lambda h_{0}(E,F)+\left(\dim\mathcal{P}-1\right)\bar{k}\bar{\mu}h_{1}(E,F)\\
& = & \lambda h_{0}(E,F)+\sigma h_{1}(E,F)\\
\sigma & = & \left(\dim\mathcal{P}-1\right)\bar{k}\bar{\mu}.
\end{eqnarray*}

The formula for the scalar curvature then follows from tracing this equation.
\end{proof}

As the scalar curvature of $g_t$ is constant, we have the following

\begin{corollary} For the metric $g_{t}$, we have
\begin{equation*}
\sum_{i}\nabla_{E_{i}}E_{i}\in\mathcal{N},\quad\mbox{ and }\quad\sum_{a}\nabla_{E_{a}}E_{a}\in\mathcal{P}.
\end{equation*}
\end{corollary}

\begin{proof}
The scalar curvature is constant on $\Sigma$. Thus, by the second
Bianchi identity,
\begin{equation*}
0=\mathrm{div}_{g_{t}}\mathrm{Ric}^{t}=\lambda\mathrm{div}_{g_{t}}h_{0}+\sigma\mathrm{div}_{g_{t}}h_{1}.
\end{equation*}
On the other hand
\begin{equation*}
0=\mathrm{div}_{g_{t}}g_{t}=\mathrm{div}_{g_{t}}h_{0}+t\mathrm{div}_{g_{t}}h_{1}.
\end{equation*}
So if $t\neq\sigma/\lambda$, then
\begin{equation*}
\mathrm{div}_{g_{t}}h_{0}=\mathrm{div}_{g_{t}}h_{1}=0.
\end{equation*}

For any vector $X\in\mathcal{N}$ we then have
\begin{eqnarray*}
0 & = & (\mathrm{div}h_{0})(X)=\sum_{\alpha}\left(\nabla_{E_{\alpha}}h_{0}\right)(E_{\alpha},X)\\
& = &
\nabla_{E_{\alpha}}\left(h_{0}(E_{\alpha},X)\right)-h_{0}\left(\nabla_{E_{\alpha}}E_{\alpha},X\right)- h_{0}(E_{\alpha},\nabla_{E_{\alpha}}X)\\
& = &
\nabla_{E_{\alpha}}\left(g_{t}(E_{\alpha},X)\right)-g_{t}\left(\nabla_{E_{\alpha}}E_{\alpha},X\right)- g_{t}(E_{i},\nabla_{E_{i}}X)\\
& = & g_{t}(\nabla_{E_{a}}E_{a},X),
\end{eqnarray*}
which shows that $\sum_{a}\nabla_{E_{a}}E_{a}\in\mathcal{P}$ when $t\neq\sigma/\lambda$. However, by continuity, this must also be true at $t=\sigma/\lambda$. Similarly, we have $\sum_{i}\nabla_{E_{i}}E_{i}\in\mathcal{N}$.
\end{proof}

The integrability of the eigen-distributions follows from a more careful
analysis of the formula for $\mathrm{Ric}^{t}$.

\begin{proof}[Proof of Theorem \ref{thmRiccidistributionintegrable}]
We consider $g_{t}$ as a variation along the $h_{1}$ direction as $g_{t}(s)=g_{t+s}$ for $s>0$. From the assumption of the Ricci curvature, we have
\begin{equation*}
\mathrm{Ric}^{t+s}(E_{i},E_{j})=\lambda\delta_{ij}\quad\mbox{and}\quad\mathrm{Ric}^{t+s}(E_{a},E_{b})=\sigma
h_{1}(E_{a},E_{b})=\frac{\sigma}{t}\delta_{ab}.
\end{equation*}
As $\mathrm{tr}_{g_{t}}h_{1}$ is constant and $\mathrm{div}_{g_{t}}h_{1}=0$, it follows from the first variation formula of the Ricci tensor, see \cite[1.174(d)]{BesseEinstein} that
\begin{equation*}
0=\frac{1}{2}\Delta_{L}^{t}h_{1},
\end{equation*}
where $\Delta_{L}^{t}$ is the Lichnerowicz Laplacian with respect to the metric $g_{t}$, see \cite[1.180b]{BesseEinstein}. In the following we compute the expansion of $\Delta_{L}^{t}$ at $t=1$. The vanishing of the coefficients will show that both distributions are integrable.

Using Koszul's formula, the connection of $g_{t}$ is determined by
\begin{eqnarray*}
g_{t}(\nabla_{X}Y,Z) & = & g_{1}(\nabla_{X}Y,Z)\\
g_{t}(\nabla_{X}Y,U) & = & g_{1}(\nabla_{X}Y,U)+\frac{t-1}{2}g_{1}([X,Y],U)\\
g_{t}(\nabla_{X}U,Y) & = & g_{1}(\nabla_{X}U,Y)+\frac{t-1}{2}g_{1}([Y,X],U)\\
g_{t}(\nabla_{X}U,V) & = & tg_{1}(\nabla_{X}U,V)+\frac{t-1}{2}g_{1}([U,V],X);
\end{eqnarray*}
and
\begin{eqnarray*}
g_{t}(\nabla_{U}X,Y) & = & g_{1}(\nabla_{U}X,Y)+\frac{t-1}{2}g_{1}([Y,X],U)\\
g_{t}(\nabla_{U}X,V) & = & tg_{1}(\nabla_{U}X,V)+\frac{t-1}{2}g_{1}([U,V],X)\\
g_{t}(\nabla_{U}V,W) & = & tg_{1}(\nabla_{U}V,W)\\
g_{t}(\nabla_{U}V,X) & = & tg_{1}(\nabla_{U}V,X)+\frac{t-1}{2}g_{1}([V,U],X).
\end{eqnarray*}

Since $\mathrm{Ric}=\lambda h_{0}+\sigma h_{1}$, we have $\mathrm{Ric}\circ h_{1}=h_{1}\circ\mathrm{Ric}=\frac{\sigma}{t}h_{1}$ as symmetric 2-tensors which implies that
\begin{equation}
\Delta_{L}^{t}h_{1}=\nabla^{\ast}\nabla h_{1}+\frac{2\sigma}{t}h_{1}- 2\dot{R}_{g_{t}}h_{1},\label{eqnLichLap}
\end{equation}
where $\dot{R}_{g}h$ is defined by
\begin{equation*}
(\dot{R}_{g_{t}}h)(A,B)=\sum_{\alpha}h(R^{t}(E_{\alpha},A)B,E_{\alpha})= \sum_{\alpha,\beta}R^{t}(E_{\alpha},A,B,E_{\beta})h(E_{\alpha},E_{\beta})
\end{equation*}
for any symmetric 2-tensor $h$ and $R^{t}$ is the (3,1) or (4,0) curvature tensor of $g_{t}$. It follows that
\begin{equation*}
\dot{R}_{g_{t}}h_{1}(A,B)=\frac{1}{t}\sum_{a}R^{t}(E_{a},A,B,E_{a}).
\end{equation*}
Let $H=\sum_{\alpha}\nabla_{E_{\alpha}}E_{\alpha}$ and then the term $\nabla^{\ast}\nabla h_{1}$ can be computed as
\begin{eqnarray*}
&  & (\nabla^{\ast}\nabla h_{1})(A,B)\\
& = & -\sum_{\alpha}\left(\nabla_{E_{\alpha},E_{\alpha}}^{2}h_{1}\right)(A,B)\\
& = & -\nabla_{E_{\alpha}}\left((\nabla h_{1})(E_{\alpha},A,B)\right)+(\nabla h_{1})(H,A,B)+(\nabla
h_{1})(E_{\alpha},\nabla_{E_{\alpha}}A,B)+(\nabla h_{1})(E_{\alpha},A,\nabla_{E_{\alpha}}B)\\
& = &
-\nabla_{E_{\alpha}}\nabla_{E_{\alpha}}(h_{1}(A,B))+2\nabla_{E_{\alpha}}(h_{1}(\nabla_{E_{\alpha}}A,B))+ 2\nabla_{E_{\alpha}}(h_{1}(A,\nabla_{E_{\alpha}}B))\\
&  & +\nabla_{H}(h_{1}(A,B))-h_{1}(\nabla_{H}A,B)-h_{1}(A,\nabla_{H}B)\\
&  &
-h_{1}(\nabla_{E_{\alpha}}\nabla_{E_{\alpha}}A,B)-2h_{1}(\nabla_{E_{\alpha}}A,\nabla_{E_{\alpha}}B)- h_{1}(A,\nabla_{E_{\alpha}}\nabla_{E_{\alpha}}B).
\end{eqnarray*}
It follows that
\begin{eqnarray*}
(\nabla^{\ast}\nabla h_{1})(X,Y) & = & -2h_{1}(\nabla_{E_{\alpha}}X,\nabla_{E_{\alpha}}Y)\\
& = & -\frac{2}{t}g_{t}(\nabla_{E_{\alpha}}X,E_{a})g_{t}(\nabla_{E_{\alpha}}Y,E_{a}),\\
(\nabla^{\ast}\nabla h_{1})(X,U) & = &
\frac{1}{t}g_{t}(\nabla_{E_{\alpha}}\nabla_{E_{\alpha}}X,U)+ \frac{2}{t}g_{t}(\nabla_{E_{\alpha}}X,E_{i})g_{t}(\nabla_{E_{\alpha}}U,E_{i})-\frac{1}{t}g_{t}(\nabla_{H}X,U),\\
(\nabla^{\ast}\nabla h_{1})(U,V) & = &
-\frac{1}{t}\nabla_{E_{\alpha}}\nabla_{E_{\alpha}}(g_{t}(U,V))+ \frac{2}{t}\nabla_{E_{\alpha}}(g_{t}(\nabla_{E_{\alpha}}U,V))+ \frac{2}{t}\nabla_{E_{\alpha}}(g_{t}(U,\nabla_{E_{\alpha}}V))\\
&  &
-\frac{1}{t}g_{t}(\nabla_{E_{\alpha}}\nabla_{E_{\alpha}}U,V)- \frac{2}{t}g_{t}(\nabla_{E_{\alpha}}U,E_{a})g_{t}(\nabla_{E_{\alpha}}V,E_{a})- \frac{1}{t}g_{t}(U,\nabla_{E_{\alpha}}\nabla_{E_{\alpha}}V)\\
& = &
\frac{2}{t}g_{t}(\nabla_{E_{\alpha}}U,\nabla_{E_{\alpha}}V)- \frac{2}{t}g_{t}(\nabla_{E_{\alpha}}U,E_{a})g_{t}(\nabla_{E_{\alpha}}V,E_{a})\\
& = & \frac{2}{t}g_{t}(\nabla_{E_{\alpha}}U,E_{i})g_{t}(\nabla_{E_{\alpha}}V,E_{i}).
\end{eqnarray*}

We choose an arbitrary vector field $U$ in $\mathcal{P}$. Since $(\Delta_{L}^{t}h_{1})(U,U)=0$, we have
\begin{equation}
\frac{t}{2}(\nabla^{\ast}\nabla h_{1})(U,U)+\sigma h_{1}(U,U)=\sum_{a}R^{t}(E_{a},U,U,E_{a}).\label{eqnLapUU}
\end{equation}

The first term on the left hand side is
\begin{eqnarray*}
\frac{t}{2}(\nabla^{\ast}\nabla h_{1})(U,U) & = &
g_{t}(\nabla_{E_{\alpha}}U,E_{i})g_{t}(\nabla_{E_{\alpha}}U,E_{i})=g_{t}(\nabla_{E_{j}}U,E_{i})^{2}+ g_{t}(\nabla_{E_{a}}U,E_{i})^{2}\\
& = & \left(g_{1}(\nabla_{E_{j}}U,E_{i})+\frac{t-1}{2}g_{1}([E_{i},E_{j}],U)\right)^{2}\\
&  & +\left(tg_{1}(\nabla_{E_{a}}U,E_{i})+\frac{t-1}{2}g_{1}([U,E_{a}],E_{i})\right)^{2}\\
& = & \left(g_{1}(\nabla_{e_{j}}U,e_{i})+\frac{t-1}{2}g_{1}([e_{i},e_{j}],U)\right)^{2}\\
&  & +\left(\sqrt{t}g_{1}(\nabla_{e_{a}}U,e_{i})+\frac{t-1}{2\sqrt{t}}g_{1}([U,e_{a}],e_{i})\right)^{2}.
\end{eqnarray*}
In particular, the highest power of $t$ in the expansion above is $t^{2}$ and the lowest one is $t^{-1}$. We call the terms with $t$'s power $0$ or $1$ the middle order terms. Then we have
\begin{equation*}
\frac{t}{2}(\nabla^{\ast}\nabla
h_{1})(U,U)=\frac{t^{2}}{4}\sum_{i,j}g_{1}\left([e_{i},e_{j}],U\right)^{2}+ \frac{1}{4t}\sum_{a,i}g_{1}\left([U,e_{a}],e_{i}\right)^{2}+\mbox{
middle order terms}.
\end{equation*}
The second term on the left hand side is
\begin{equation*}
\sigma h_{1}(U,U)=\sigma g_{1}(U,U).
\end{equation*}
Therefore the expansion of the left hand side is given by
\begin{equation}
\left(\frac{t}{2}(\nabla^{\ast}\nabla h_{1})+\sigma
h_{1}\right)(U,U)=\frac{t^{2}}{4}\sum_{i,j}g_{1}\left([e_{i},e_{j}],U\right)^{2}+ \frac{1}{4t}\sum_{a,i}g_{1}\left([U,e_{a}],e_{i}\right)^{2}+\mbox{
middle order terms}.\label{eqnexpansionLHS}
\end{equation}

The right hand side is
\begin{eqnarray*}
\sum_{a}R^{t}(E_{a},U,U,E_{a}) & = &
D_{E_{a}}g_{t}(\nabla_{U}U,E_{a})-g_{t}(\nabla_{U}U,\sum_{a}\nabla_{E_{a}}E_{a})- D_{U}g_{t}(\nabla_{E_{a}}U,E_{a})\\
&  & +g_{t}(\nabla_{E_{a}}U,\nabla_{U}E_{a})-g_{t}(\nabla_{\lbrack E_{a},U]}U,E_{a}).
\end{eqnarray*}
We calculate each term in $R^{t}(E_{a},U,U,E_{a})$ and only consider
the term $t^{2}$ and $t^{-1}$. We have

\begin{eqnarray*}
D_{E_{a}}g_{t}(\nabla_{U}U,E_{a}) & = &
D_{E_{a}}\left(tg_{1}(\nabla_{U}U,E_{a})\right)=D_{e_{a}}g_{1}(\nabla_{U}U,e_{a})=\mbox{ middle order terms},\\
g_{t}(\nabla_{U}U,\nabla_{E_{a}}E_{a}) & = &
g_{t}(\nabla_{U}U,E_{b})g_{t}(\nabla_{E_{a}}E_{a},E_{b})=tg_{1}(\nabla_{U}U,E_{b})\cdot
tg_{1}(\nabla_{E_{a}}E_{a},E_{b})\\
& = & g_{1}(\nabla_{U}U,e_{b})g_{1}(\nabla_{e_{a}}e_{a},e_{b})=g_{1}(\nabla_{U}U,\nabla_{e_{a}}e_{a})=\mbox{
middle order terms},\\
D_{U}g_{t}(\nabla_{E_{a}}U,E_{a}) & = &
D_{U}\left(tg_{1}(\nabla_{E_{a}}U,E_{a})\right)=D_{U}g_{1}(\nabla_{e_{a}}U,e_{a})=\mbox{ middle order terms};
\end{eqnarray*}
and
\begin{eqnarray*}
g_{t}(\nabla_{E_{a}}U,\nabla_{U}E_{a}) & = &
g_{t}(\nabla_{E_{a}}U,E_{i})g_{t}(\nabla_{U}E_{a},E_{i})+ g_{t}(\nabla_{E_{a}}U,E_{b})g_{t}(\nabla_{U}E_{a},E_{b})\\
& = & \left(tg_{1}(\nabla_{E_{a}}U,E_{i})+\frac{t-1}{2}g_{1}([U,E_{a}],E_{i})\right)\cdot\\
&  & \left(tg_{1}(\nabla_{U}E_{a},E_{i})+\frac{t-1}{2}g_{1}([E_{a},U],E_{i})\right)\\
&  & +tg_{1}(\nabla_{E_{a}}U,E_{b})\cdot tg_{1}(\nabla_{U}E_{a},E_{b})\\
& = & \left(\sqrt{t}g_{1}(\nabla_{e_{a}}U,e_{i})+\frac{t-1}{2\sqrt{t}}g_{1}([U,e_{a}],e_{i})\right)\cdot\\
&  & \left(\sqrt{t}g_{1}(\nabla_{U}e_{a},e_{i})+\frac{t-1}{2\sqrt{t}}g_{1}([e_{a},U],e_{i})\right)\\
&  & +g_{1}(\nabla_{e_{a}}U,e_{b})\cdot g_{1}(\nabla_{U}e_{a},e_{b})\\
& = & -\frac{1}{4t}\sum_{a,i}g_{1}\left([U,e_{a}],e_{i}\right)^{2}+\mbox{ middle order terms},\\
g_{t}(\nabla_{\lbrack E_{a},U]}U,E_{a}) & = &
g_{t}(\nabla_{\mathcal{N}[E_{a},U]}U,E_{a})+g_{t}(\nabla_{\mathcal{P}[E_{a},U]}U,E_{a})\\
& = &
tg_{1}(\nabla_{\mathcal{N}[E_{a},U]}U,E_{a})+\frac{t-1}{2}g_{1}([U,E_{a}],\mathcal{N}[E_{a},U])+ tg_{1}(\nabla_{\mathcal{P}[E_{a},U]}U,E_{a})\\
& = &
g_{1}(\nabla_{\mathcal{N}[e_{a},U]}U,e_{a})+\frac{t-1}{2t}g_{1}([U,e_{a}],\mathcal{N}[e_{a},U])+ g_{1}(\nabla_{\mathcal{P}[e_{a},U]}U,e_{a})\\
& = & \frac{1}{2t}\sum_{a,i}g_{1}\left([U,e_{a}],e_{i}\right)^{2}+\mbox{ middle order terms.}
\end{eqnarray*}
It follows that
\begin{equation}
\sum_{a}R^{t}(E_{a},U,U,E_{a})=-\frac{3}{4t}\sum_{a,i}g_{1}\left([U,e_{a}],e_{i}\right)^{2}+\mbox{ middle order
terms.}.\label{eqnexpansionRh1}
\end{equation}
Since the two expansions given by (\ref{eqnexpansionLHS}) and (\ref{eqnexpansionRh1})
in the variable $t$ are equal and $U$ is an arbitrary vector field
in $\mathcal{P}$, both distributions $\mathcal{N}$ and $\mathcal{P}$
are integrable.
\end{proof}

\medskip{}

\medskip{}

\end{document}